\documentclass[a4paper]{article}

\usepackage{graphicx}
\usepackage{color}
\usepackage{stmaryrd}
\usepackage[all]{xy}

\usepackage[numbers]{natbib}
\usepackage{natbibspacing}
\renewcommand{\bibfont}{\small}

\definecolor{myurlcolor}{rgb}{0.1,0.1,0.8}
\definecolor{mylinkcolor}{rgb}{0.05,0.05,0.4}
\usepackage{hyperref}
\hypersetup{colorlinks,
linkcolor=mylinkcolor,
citecolor=mylinkcolor,
urlcolor=mylinkcolor}

\usepackage{tocloft}
\usepackage{latexsym}
\usepackage{amssymb,amsmath}
\usepackage{mathrsfs}
\usepackage{color}
\newcommand{\bref}[1]{(\ref{#1})}
\newcommand{\cat}[1]{\mathscr{#1}}
\newcommand{\fcat}[1]{\mathbf{#1}}
\newcommand{\such}{:}
\newcommand{\integers}{\mathbb{Z}}
\newcommand{\lbl}[1]{\label{#1}}
\newcommand{\reals}{\mathbb{R}}

\newcommand{\demph}[1]{\textbf{\textup{#1}}}
\newcommand{\iso}{\cong}
\newcommand{\sub}{\subseteq}
\newcommand{\nat}{\mathbb{N}}	
\newcommand{\of}{\circ}
\newcommand{\toby}[1]{\stackrel{#1}{\longrightarrow}}
\newcommand{\mg}[1]{|#1|}
\newcommand{\from}{\colon}

\newcommand{\R}{\reals}
\newcommand{\Z}{\integers}
\newcommand{\Q}{\mathbb{Q}}
\newcommand{\vc}[1]{\mathbf{#1}} % Vector
\DeclareMathOperator{\supp}{supp}
\newcommand{\csuch}{\colon}
\newcommand{\hardref}[1]{#1}
\newcommand{\hlf}{\tfrac{1}{2}}
\newcommand{\cpd}{\sqcup}
\newcommand{\ppi}{{\boldsymbol{\pi}}}
\newcommand{\ggamma}{{\boldsymbol{\gamma}}}
\newcommand{\ssigma}{{\boldsymbol{\sigma}}}
\newcommand{\nnu}{{\boldsymbol{\nu}}}
\newcommand{\ttau}{{\boldsymbol{\tau}}}
\newcommand{\fld}{K}
\newcommand{\fq}[1]{q_{#1}}
\newcommand{\Zp}{\Z/p\Z}
\newcommand{\Zps}{\Z/p^2\Z}
\newcommand{\Ztwo}{\Z/2\Z}
\newcommand{\dvd}{\mathrel{\mid}}
\newcommand{\ndvd}{\mathrel{\nmid}}
\newcommand{\FinProbp}{\fcat{FinProb}_p}

\newcommand{\Rat}[2]{\Delta_{#1, #2}}
\newcommand{\Rentp}{\mathbb{E}_p}

\usepackage[amsmath,thmmarks]{ntheorem}

\newtheorem{thm}{Theorem}[section]
\newtheorem{propn}[thm]{Proposition}
\newtheorem{lemma}[thm]{Lemma}
\newtheorem{cor}[thm]{Corollary}

\theorembodyfont{\normalfont}

\newtheorem{defn}[thm]{Definition}
\newtheorem{example}[thm]{Example}

\newtheorem{remark}[thm]{Remark}

\theoremstyle{nonumberplain}
\theoremsymbol{\ensuremath{\square}}
\qedsymbol{\ensuremath{\square}}

\newtheorem{proof}{Proof}

\newcommand{\theoremtobeproved}{}
\newtheorem{pfoftheorem}{Proof of \theoremtobeproved}
\newenvironment{pfof}[1]
{
\renewcommand{\theoremtobeproved}{#1}
\begin{pfoftheorem}
}
{\end{pfoftheorem}}

\title{Entropy modulo a prime}
\author{Tom Leinster%
\thanks{School of Mathematics, University of Edinburgh, Scotland; email
  Tom.Leinster@ed.ac.uk.  MSC 2010: 94A17 (primary), 11A07, 11A99, 11T06,
  13N15.}}
\date{\vspace{-3ex}}

\begin{document}

\sloppy
\maketitle

\begin{abstract}
Building on work of Kontsevich, we introduce a definition of the entropy of
a finite probability distribution in which the `probabilities' are integers
modulo a prime~$p$.  The entropy, too, is an integer mod~$p$.  Entropy
mod~$p$ is shown to be uniquely characterized by a functional equation
identical to the one that characterizes ordinary Shannon entropy.  We also
establish a sense in which certain real entropies have residues mod~$p$,
connecting the concepts of entropy over $\R$ and over $\Zp$.  Finally,
entropy mod~$p$ is expressed as a polynomial which is shown to satisfy
several identities, linking into work of Cathelineau, Elbaz-Vincent and
Gangl on polylogarithms.
\end{abstract}

\tableofcontents

\section{Introduction}

The concept of entropy is applied in almost every branch of science.  Less
widely appreciated, however, is that from a purely algebraic perspective,
entropy has a very simple nature.  Indeed, Shannon entropy is characterized
nearly uniquely by a single equation, expressing a recursivity property.
The purpose of this work is to introduce a parallel notion of entropy for
probability distributions whose `probabilities' are not real numbers, but
integers modulo a prime $p$.  The entropy of such a distribution is also an
integer mod~$p$.

We will see that despite the (current) lack of scientific application, this
`entropy' is fully deserving of the name.  Indeed, it is characterized by a
recursivity equation formally identical to the one that characterizes
classical real entropy.  It is also directly related to real entropy, via a
notion of residue informally suggested by Kontsevich~\cite{KontOHL}.

Among the many types of entropy, the most basic is the Shannon entropy of a
finite probability distribution $\ppi = (\pi_1, \ldots, \pi_n)$, defined as
\begin{align}
H(\ppi) = - \sum_{i \csuch \pi_i \neq 0} \pi_i \log \pi_i 
\in \R.
\end{align}
It is this that we will imitate in the mod~$p$ setting.  

The aforementioned recursivity property concerns the entropy of the
composite of two processes, in which the nature of the second process
depends on the outcome of the first.  Specifically, let $\ppi =
(\pi_1, \ldots, \pi_n)$ be a finite probability distribution, and let
$\ggamma^1, \ldots, \ggamma^n$ be further distributions, writing $\ggamma^i =
(\gamma^i_1, \ldots, \gamma^i_{k_i})$.  Their \demph{composite} is
\begin{align}
\lbl{eq:comp}
\ppi \of (\ggamma^1, \ldots, \ggamma^n)
=
(\pi_1 \gamma^1_1, \ldots, \pi_1 \gamma^1_{k_1},
\ \ldots,\ 
\pi_n \gamma^n_1, \ldots, \pi_n \gamma^n_{k_n}),
\end{align}
a probability distribution on $k_1 + \cdots + k_n$ elements.  (Formally,
this composition endows the sequence of simplices $(\Delta^{n - 1})_{n =
  0}^\infty$ with the structure of an operad.)  The \demph{chain rule} (or
\demph{recursivity} or \demph{grouping} law) for Shannon entropy is that
\begin{align}
\lbl{eq:chain-rule}
H\bigl( \ppi \of (\ggamma^1, \ldots, \ggamma^n) \bigr)=
H(\ppi) + \sum_{i = 1}^n \pi_i H(\ggamma^i).
\end{align}

The chain rule can be understood in terms of information.  Suppose we toss
a fair coin and then, depending on the outcome, either roll a fair die or
draw fairly from a pack of~52 cards.  There are $6 + 52 = 58$ possible
final outcomes, and their probabilities are given by the composite
distribution
\begin{align}
\bigl(\hlf, \hlf\bigr) \of 
\Bigl( 
\bigl(\underbrace{\tfrac{1}{6}, \ldots, \tfrac{1}{6}}_6\bigr), \ 
\bigl(\underbrace{\tfrac{1}{52}, \ldots, \tfrac{1}{52}}_{52}\bigr) 
\Bigr)   
=
\bigl(\underbrace{\tfrac{1}{12}, \ldots, \tfrac{1}{12}}_6, 
\underbrace{\tfrac{1}{104}, \ldots, \tfrac{1}{104}}_{52}\bigr).
\end{align}
Now, the entropy of a distribution $\ppi$ measures the amount of
information gained by learning the outcome of an observation drawn from
$\ppi$ (measured in bits, if logarithms are taken to base $2$).  In our
example, knowing the outcome of the composite process tells us
with certainty the outcome of the initial coin toss, plus with probability
$1/2$ the outcome of a die roll and with probability $1/2$ the outcome of a
card draw.  Thus, the entropy of the composite distribution should be equal
to 
\begin{align}
H\bigl(\hlf, \hlf\bigr)
+
\hlf H\bigl(\tfrac{1}{6}, \ldots, \tfrac{1}{6}\bigr)
+
\hlf H\bigl(\tfrac{1}{52}, \ldots, \tfrac{1}{52}\bigr).
\end{align}
This is indeed true, and is an instance of the chain rule.

A classical theorem essentially due to Faddeev~\cite{Fadd} states that up
to a constant factor, Shannon entropy $H$ is the only continuous function
assigning a real number to each finite probability distribution in such a
way that the chain rule holds.  In this sense, the chain rule is the
characteristic property of entropy.

Our first task will be to formulate the right definition of entropy
mod~$p$.  An immediate obstacle is that there is no logarithm function
mod~$p$, at least in the most obvious sense.  Nevertheless, the classical
Fermat quotient turns out to provide an acceptable substitute
(Section~\ref{sec:log-der}).  Closely related to the real logarithm is the
nonlinear derivation $\partial \from x \mapsto -x \log x$, and its mod~$p$
analogue is $\partial \from x \mapsto (x - x^p)/p$ (a
\demph{$p$-derivation}, in the language of Buium~\cite{BuiuDCA}).

The entropy of a mod~$p$ probability distribution $\ppi = (\pi_1, \ldots,
\pi_n)$, with $\pi_i \in \Zp$, is then defined as
\begin{align}
H(\ppi) 
=
\sum \partial(a_i) - \partial\Bigl( \sum a_i \Bigr)
=
\frac{1}{p} \biggl( 1 - \sum a_i^p \biggr)
\in 
\Zp,
\end{align}
where $a_i$ is an integer representing $\pi_i \in \Zp$
(Section~\ref{sec:defn-ent}).  The definition is independent of the choice
of representatives $a_i$.  This entropy satisfies a chain rule formally
identical to that satisfied by real entropy (Section~\ref{sec:properties}).
We prove in Section~\ref{sec:char-ent} that up to a constant factor, $H$ is
the one and only function satisfying the chain rule.  This is the main
justification for the definition.

Classical Shannon entropy quantifies the information associated with a
probability space, but one can also seek to quantify the information lost
by a \emph{map} between probability spaces, seen as a deterministic
process.  For example, if one chooses uniformly at random a binary number
with ten digits, then discards the last three, the discarding process loses
three bits.  

There is a formal definition of information loss, it includes the
definition of entropy as a special case, and it has been uniquely
characterized in work of Baez, Fritz and Leinster~\cite{CETIL}.  The
advantage of working with information loss rather than entropy is that the
characterizing equations look exactly like the linearity and homomorphism
conditions that occur throughout algebra---in contrast to the chain rule.
In Section~\ref{sec:loss}, we show that an analogous characterization
theorem holds mod~$p$.

We then make precise an idea of Kontsevich linking entropy over $\R$ with
entropy over $\Zp$.  Consider a distribution $\ppi$ whose probabilities
$\pi_i$ are rational numbers.  On the one hand, we can take its real
entropy $H_\R(\ppi)$.  On the other, whenever $p$ is a prime not dividing
the denominator of any $\pi_i$, we can view $\ppi$ as a probability
distribution mod~$p$ and therefore take its entropy $H_p(\ppi)$ mod~$p$.
Kontsevich suggested viewing $H_p(\ppi) \in \Zp$ as the `residue' of
$H_\R(\ppi) \in \R$, and Section~\ref{sec:res} establishes that this
construction has the basic properties that one would expect from the name.

Finally, we analyse $H$ not as a function but as a polynomial
(Sections~\ref{sec:poly-defn} and~\ref{sec:poly-ids}).  We show that
\begin{align}
H(\ppi)
=
-\sum_{\substack{0 \leq r_1, \ldots, r_n < p\\ r_1 + \cdots + r_n = p}}
\frac{\pi_1^{r_1} \cdots \pi_n^{r_n}}{r_1! \cdots r_n!}
\end{align}
(which \emph{formally} is equal to $\tfrac{1}{p}(\sum \pi_i^p - (\sum
\pi_i)^p)$).  We prove several identities in this polynomial.  In the case
of distributions $(\pi, 1 - \pi)$ on two elements, we find that
\begin{align}
\lbl{eq:poly-two}
H(\pi, 1 - \pi) 
=
\sum_{r = 1}^{p - 1} \frac{\pi^r}{r}
\end{align}
for $p \neq 2$, and we discuss some properties that this polynomial
possesses.

The present work should be regarded as a beginning rather than an end.  In
information theory, Shannon entropy is just the simplest of a family of
fundamental concepts including relative entropy, conditional entropy, and
mutual information.  It is natural to seek their mod~$p$ analogues, and to
prove analogous theorems; however, this is not attempted here.

\paragraph{Related work}
This work builds on a two-and-a-half page note of
Kontsevich~\cite{KontOHL}.  In it, Kontsevich did just enough to show that
a reasonable definition of entropy mod~$p$ must exist, but without actually
giving the definition except for probability distributions on two elements.
He also briefly suggested viewing the entropy mod $p$ of a distribution
with rational probabilities as the `residue' of its real entropy.  The
relationship between his note and the present work is further clarified at
the start of Section~\ref{sec:res} and the end of
Section~\ref{sec:poly-ids}.

Kontsevich's note appears to have been motivated by questions about
polylogarithms.  (The polynomial~\eqref{eq:poly-two} is a truncation of the
power series of $-\log(1 - \pi)$, and one can consider more generally a
truncation $\sum_{r = 1}^{p - 1} \pi^r/r^m$ of the $m$th polylogarithm.)
That line of enquiry has been pursued by Elbaz-Vincent and
Gangl~\cite{EVGOPI,EVGFPM}.  As recounted in the introduction
to~\cite{EVGOPI}, some of Kontsevich's results had already appeared in
papers of Cathelineau~\cite{CathSHS,CathRDP}.  The connection between this
part of algebra and information theory was noticed at least as far back as
1996 (\cite{CathRDP}, p.~1327).  In the present work, however,
polylogarithms play no part and entropy takes centre stage.

A fully-fledged theory of information cohomology has been introduced by
Baudot and Bennequin~\cite{BaBe} and extended in several directions by
Vigneaux~\cite{VignTSS}; it concerns topos invariants of categories of
random variables.  A basic result is that Shannon entropy is the only
nontrivial cohomology class in degree $1$ (for a suitable choice of
coefficients).  The characterization below of entropy mod $p$ can also be
understood in terms of degree $1$ information cohomology, over 
$\Zp$. 

Unlike much previous work on characterizations of entropies, we are able to
do without symmetry axioms.  (Nor is symmetry used in information
cohomology, as noted after Theorem~1 in~\cite{BaBe}.)  For example,
Faddeev's theorem on real entropy~\cite{Fadd} characterized it as the
unique continuous quantity satisfying the chain rule \emph{and} invariant
under permutation of its arguments.  However, a careful reading of the
proof shows that the symmetry assumption can be dropped.  The
axiomatization of entropy via the so-called fundamental equation of
information theory also uses a symmetry assumption.  While symmetry appears
to be essential to \emph{that} approach (Remark~\ref{rmk:k}), we will not
need it.

The chain rule~\eqref{eq:chain-rule} is often stated in the case $k_1 = 2$,
$k_2 = \cdots = k_n = 1$, or occasionally in the different case $n = 2$.
In the presence of the symmetry axiom, either of these cases implies the
general case, by induction.  For example, Faddeev used the first 
case, whose asymmetry forced him to add the symmetry axiom; but that
can be avoided by assuming the chain rule in its general form.

The operation $\partial \from a \mapsto (a - a^p)/p$ mentioned above is
basic in the theory of $p$-derivations (as in Buium~\cite{BuiuDCA}), which are
themselves closely related to Frobenius lifts and the Adams operations on
$K$-theory (as in Joyal~\cite{JoyDAL}).

One can speculate about extending the theory of entropy to fields other
than $\R$ and $\Zp$, and in particular to the $p$-adic numbers.  (The
$p$-adic entropy of Deninger~\cite{DenipEp} is of a different nature.)
Again, there may be a connection with the information cohomology of Baudot,
Bennequin and Vigneaux, which takes place over an arbitrary field.

\paragraph{Convention}
Throughout, $p$ denotes a prime number, possibly $2$.

\paragraph{Acknowledgements} 
I thank James Borger, Herbert Gangl and Todd Trimble for enlightening
conversations.

\section{Logarithms and derivations}
\lbl{sec:log-der}

Real entropy is a kind of higher logarithm, in the senses that it has the
multiplication-to-addition property
\begin{align}
H(\ppi \otimes \ggamma)
= 
H(\ppi) + H(\ggamma)
\end{align}
(in notation defined at the end of Section~\ref{sec:properties}), and that
when restricted to uniform distributions, it is the logarithm function
itself:
\begin{align}
H(1/n, \ldots, 1/n) = \log n.
\end{align}
To find the right definition of entropy mod~$p$, we therefore begin by
considering mod $p$ notions of logarithm.

Lagrange's theorem immediately implies that there is no logarithm mod $p$,
in that the only homomorphism from the multiplicative group $(\Zp)^\times$
to the additive group $\Zp$ is trivial.  However, there is a substitute.
For an integer $a$ not divisible by $p$, the \demph{Fermat quotient} of $a$
mod~$p$ is the integer
\begin{align}
\fq{p}(a) = \frac{a^{p - 1} - 1}{p}.
\end{align}
We usually regard $\fq{p}(a)$ as an element of $\Zp$.  
Eisenstein~\cite{Eise} observed:

\begin{lemma}
\lbl{lemma:fq-elem}
The map $\fq{p} \from \{ n \in \Z \such p \ndvd n \} \to \Zp$ has the
following properties:
\begin{enumerate}
\item 
\lbl{part:fq-elem-log} 
$\fq{p}(mn) = \fq{p}(m) + \fq{p}(n)$ for all $m, n \in
\Z$ not divisible by~$p$, and $\fq{p}(1) = 0$;

\item
\lbl{part:fq-elem-shift}
$\fq{p}(n + rp) = \fq{p}(n) - r/n$ for all $n, r \in \Z$ with $n$ not
divisible by $p$;

\item
\lbl{part:fq-elem-per}
$\fq{p}(n + p^2) = \fq{p}(n)$ for all $n \in \Z$ not divisible by $p$.
\end{enumerate}
\end{lemma}

\begin{proof}
Elementary calculations using Fermat's little theorem.
\end{proof}

The lemma implies that $\fq{p}$ defines a group homomorphism
\begin{align}
\fq{p} \from (\Zps)^\times \to \Zp.
\end{align}
It is surjective, since by
the lemma again, it has a section $r \mapsto 1 - rp$.

The Fermat quotient is the closest approximation to a logarithm mod~$p$, in
the sense that although there is no nontrivial group homomorphism
$(\Zp)^\times \to \Zp$, it is a homomorphism $(\Zps)^\times \to \Zp$.  It
is essentially unique as such:

\begin{propn}
\lbl{propn:fq-mult}
Every group homomorphism $(\Zps)^\times \to \Zp$ is a scalar multiple of
the Fermat quotient.
\end{propn}

\begin{proof}
This follows from the standard fact that the group $(\Zps)^\times$ is cyclic
(Theorem~10.6 of Apostol~\cite{AposIAN}, for instance), together with the
observation that $\fq{p}$ is nontrivial (being surjective).  Indeed, let $e$
be a generator of $(\Zps)^\times$; then given
$\phi \from (\Zps)^\times \to \Zp$, we have $\phi = c\fq{p}$ where $c =
\phi(e)/\fq{p}(e) \in \Zp$. 
\end{proof}

Our characterization theorem for entropy mod~$p$ will use the following
characterization of the Fermat quotient.

\begin{propn}
\lbl{propn:fq-char}
Let $f \from \{ n \in \nat \such p \ndvd n \} \to \Zp$ be a function.  The
following are equivalent:
\begin{enumerate}
\item 
\lbl{part:fq-char-condns}
$f(mn) = f(m) + f(n)$ and $f(n + p^2) = f(n)$ for all $m, n \in \nat$ not
divisible by $p$;

\item
\lbl{part:fq-char-form}
$f = c\fq{p}$ for some $c \in \Zp$.
\end{enumerate}
\end{propn}

\begin{proof}
Since $f = \fq{p}$ satisfies the conditions
in~\bref{part:fq-char-condns}, so does any constant multiple.
Hence~\bref{part:fq-char-form} implies~\bref{part:fq-char-condns}.
The converse follows from Proposition~\ref{propn:fq-mult}.
\end{proof}

The entropy of a real probability distribution $\ppi = (\pi_1, \ldots,
\pi_n)$ is 
\begin{align}
H_\R(\ppi) = \sum_{i = 1}^n \partial_\R(\pi_i),
\end{align}
where
\begin{align}
\lbl{eq:real-deriv}
\partial_\R(x) = 
\begin{cases}
-x\log x        &\text{if } x > 0,   \\
0               &\text{if } x = 0.
\end{cases}
\end{align}
The operator $\partial_\R$ is a nonlinear derivation, in the sense that
\begin{align}
\partial_\R(xy) = \partial_\R(x)y + x\partial_\R(y),
\qquad
\partial_\R(1) = 0.
\end{align}
In particular, $\partial_\R(\sum \pi_i) = 0$.  The entropy of $\ppi$
therefore measures the failure of the nonlinear operator $\partial_\R$ to
preserve the sum $\sum \pi_i$:
\begin{align}
\lbl{eq:real-ent-nonlin}
H_\R(\ppi) = 
\sum_{i = 1}^n \partial_\R(\pi_i)
- \partial_\R \Biggl( \sum_{i = 1}^n \pi_i \Biggr).
\end{align}
We will define entropy mod~$p$ in such a way that the analogue of this
equation holds.

The mod~$p$ analogue of $\partial_\R$ is the function $\partial_p \from \Z
\to \Z$ defined by
\begin{align}
\partial_p(a) 
= 
\frac{a - a^p}{p}
=
\begin{cases}
-a \fq{p}(a)    &\text{if } p \ndvd a,  \\
a/p             &\text{if } p \dvd a.
\end{cases}
\end{align}
We usually abbreviate $\partial_p$ to $\partial$, and treat
$\partial(a)$ as an integer mod $p$.  Evidently the element $\partial(a)$
of $\Zp$ depends only on the residue class of $a$ mod~$p^2$, so we can also
view $\partial$ as a function $\Zps \to \Zp$.

\begin{lemma}
\lbl{lemma:p-deriv-elem}
$\partial(ab) \equiv \partial(a)b + a\partial(b) \pmod{p}$ for all $a, b
\in \Z$, and $\partial(1) = 0$.
\end{lemma}

\begin{proof}
This is an elementary consequence of Fermat's little theorem.
\end{proof}

\section{The definition of entropy}
\lbl{sec:defn-ent}

For $n \geq 1$, write
\begin{align}
\Pi_n
=
\{ \ppi \in (\Zp)^n \such \pi_1 + \cdots + \pi_n = 1\}.
\end{align}
An element of $\Pi_n$ will be called a \demph{probability distribution
  mod~$p$}, or simply a \demph{distribution}.  We will define the entropy
$H_p(\ppi) \in \Zp$ of any such distribution.

A standard elementary lemma will be repeatedly
useful:

\begin{lemma}
\lbl{lemma:pps}
Let $a, b \in \Z$.  If $a \equiv b \pmod{p}$ then $a^p \equiv b^p
\pmod{p^2}$. 
\end{lemma}

\begin{proof}
Write $b = a + rp$ and expand using the binomial theorem.
\end{proof}

The observations at the end of Section~\ref{sec:log-der} suggest defining
entropy mod~$p$ by the analogue of equation~\eqref{eq:real-ent-nonlin},
replacing $\partial_\R$ by $\partial_p$.  In principle this is impossible,
as $\partial_p$ is only well-defined on congruence classes mod~$p^2$,
not mod~$p$.  Thus, for $\pi_i \in \Zp$, the term $\partial_p(\pi_i)$ is
not well-defined.  Nevertheless, the strategy can be implemented:

\begin{lemma}
\lbl{lemma:ent-partial}
For all $n \geq 1$ and $a_1, \ldots, a_n \in \Z$ such that $\sum a_i \equiv
1 \pmod{p}$,
\begin{align}
\sum_{i = 1}^n \partial(a_i) - \partial \Biggl( \sum_{i = 1}^n a_i \Biggr)
\equiv
\frac{1}{p} \Biggl( 1 - \sum_{i = 1}^n a_i^p \Biggr)
\pmod{p}.
\end{align}
\end{lemma}

\begin{proof}
The right-hand side is an integer, since $\sum a_i^p \equiv
\sum a_i \equiv 1 \pmod{p}$.  The lemma is equivalent to the congruence
\begin{align}
\sum (a_i - a_i^p) - \biggl\{ \sum a_i - \Bigl( \sum a_i \Bigr)^p \biggr\}
\equiv
1 - \sum a_i^p
\pmod{p^2}.
\end{align}
Cancelling, this reduces to
\begin{align}
\Bigl( \sum a_i \Bigr)^p \equiv 1 \pmod{p^2}.
\end{align}
But $\sum a_i \equiv 1 \pmod{p}$, so $\bigl( \sum a_i \bigr)^p \equiv 1
\pmod{p^2}$ by Lemma~\ref{lemma:pps}.
\end{proof}

\begin{defn}
\lbl{defn:ent-p} 
Let $n \geq 1$ and $\ppi \in \Pi_n$.  The \demph{entropy} of $\ppi$ is
\begin{align}
H_p(\ppi) 
= 
\frac{1}{p} \biggl( 1 - \sum_{i = 1}^n a_i^p \biggr)
\in \Zp,
\end{align}
where $a_i \in \Z$ represents $\pi_i \in \Zp$.  We often abbreviate $H_p$ to
$H$. 
\end{defn}

Lemma~\ref{lemma:pps} guarantees that the definition is independent of the
choice of representatives $a_1, \ldots, a_n$, and
Lemma~\ref{lemma:ent-partial} gives
\begin{align}
\lbl{eq:H-partial}
H_p(\ppi) = \sum \partial_p(a_i) - \partial_p \Bigl( \sum a_i \Bigr),
\end{align}
as in the real case (equation~\eqref{eq:real-ent-nonlin}).  But in contrast
to the real case, the term $\partial_p\bigl(\sum a_i\bigr)$ is not always
zero, and if it were omitted then the right-hand side would no longer be
independent of the choice of integers $a_i$.

\begin{example}
\lbl{eg:ent-p-ufm}
Let $n \geq 1$ with $p \ndvd n$.  Then there is a \demph{uniform
  distribution}
\begin{align}
\vc{u}_n = (\underbrace{1/n, \ldots, 1/n}_n) \in \Pi_n.
\end{align}
Choose $a \in \Z$ representing $1/n \in \Zp$.  By
equation~\eqref{eq:H-partial} and then the derivation property of
$\partial$,
\begin{align}
H_p(\vc{u}_n)
=
n\partial(a) - \partial(na)
=
-a\partial(n).
\end{align}
But $\partial(n) = -n\fq{p}(n)$, so $H_p(\vc{u}_n) = \fq{p}(n)$.  This
result over $\Zp$ is analogous to the formula $H_\R(\vc{u}_n) = \log n$ for
the real entropy of a uniform distribution.
\end{example}

\begin{example}
\lbl{eg:ent-p-2}
Let $p = 2$.  For $\ppi \in \Pi_n$, write $\supp(\ppi) = \{ i \such \pi_i
\neq 0\}$, which has odd cardinality since $\sum \pi_i = 1$.  Directly from
the definition of entropy, $H(\ppi) \in \Z/2\Z$ is given by
\begin{align}
H(\ppi) 
= 
\hlf\bigl(\mg{\supp(\ppi)} - 1\bigr)   
=
\begin{cases}
0       &\text{if } \mg{\supp(\ppi)} \equiv 1 \!\!\!\!\pmod{4}, \\
1       &\text{if } \mg{\supp(\ppi)} \equiv 3 \!\!\!\!\pmod{4}.
\end{cases}
\end{align}
\end{example}

In preparation for the next example, we record a standard lemma:

\begin{lemma}
\lbl{lemma:p-binom}
$\binom{p - 1}{s} \equiv (-1)^s \pmod{p}$ for all $s \in \{0, \ldots, p -
  1\}$. 
\end{lemma}

\begin{proof}
$\binom{p - 1}{s} = 
\frac{(p - 1) \cdots (p - s)}{s!} \equiv 
\frac{(-1)^s s!}{s!} = 
(-1)^s \pmod{p}$.  
\end{proof}

\begin{example}
\lbl{eg:p-bin} 
We compute the entropy of a distribution $(\pi, 1 -
\pi)$ on two elements.  Choose $a \in \Z$ representing $\pi \in \Zp$.
Directly from the definition of entropy, and assuming that $p \neq 2$,
\begin{align}
H(\pi, 1 - \pi)
=
\frac{1}{p} \bigl( 1 - a^p - (1 - a)^p \bigr)
=
\sum_{r = 1}^{p - 1} (-1)^{r + 1} \frac{1}{p} \binom{p}{r} a^r.
\end{align}
But $\tfrac{1}{p} \binom{p}{r} = \tfrac{1}{r} \binom{p - 1}{r - 1}$, so by
Lemma~\ref{lemma:p-binom}, the coefficient of $a^r$ in the sum is $1/r$.
Hence
\begin{align}
H(\pi, 1 - \pi) = \sum_{r = 1}^{p - 1} \frac{\pi^r}{r}.
\end{align}
The function on the right-hand side was the starting point of Kontsevich's
note~\cite{KontOHL}, and we return to it in Section~\ref{sec:poly-ids}.  
In the case $p = 2$, we have $H(\pi, 1 - \pi) = 0$ for both values of $\pi
\in \Ztwo$.
\end{example}

\begin{example}
Appending zero probabilities to a distribution does not change its entropy:
\begin{align}
H(\pi_1, \ldots, \pi_n, 0, \ldots, 0) = H(\pi_1, \ldots, \pi_n).
\end{align}
A subtlety of distributions mod~$p$, absent in the standard real setting,
is that nonzero `probabilities' can sum to zero. But in general, when $\sum
\tau_j = 0$,
\begin{align}
H(\pi_1, \ldots, \pi_n, \tau_1, \ldots, \tau_m) 
\neq
H(\pi_1, \ldots, \pi_n).
\end{align}
For example, when $p = 3$, $\ppi = (1)$ and $\ttau = (1, 1, 1)$, 
Example~\ref{eg:ent-p-ufm} gives 
\begin{align}
H(1, 1, 1, 1) = \fq{3}(4) = -1 \neq 0 = H(1).
\end{align}
\end{example}

\section{The chain rule}
\lbl{sec:properties}

Here we formulate the mod~$p$ version of the chain rule for entropy, which
will later be shown to characterize entropy uniquely up to a constant.

In the Introduction, it was noted that real probability distributions can
be composed in a way that corresponds to performing two random processes in
sequence.  The same formula~\eqref{eq:comp} defines a composition of
probability distributions mod~$p$, where now
\begin{align}
\ppi \in \Pi_n,
\quad
\ggamma^i \in \Pi_{k_i}, 
\quad
\ppi \of (\ggamma^1, \ldots, \ggamma^n) \in \Pi_{k_1 + \cdots + k_n}.
\end{align}
And entropy mod~$p$ satisfies the same chain rule for composition:

\begin{propn}[Chain rule]
\lbl{propn:chn-p}
We have
\begin{align}
\lbl{eq:chn-p-main}
H_p\bigl( \ppi \of (\ggamma^1, \ldots, \ggamma^n) \bigr)
=
H_p(\ppi) + \sum_{i = 1}^n \pi_i H_p(\ggamma^i)
\end{align}
for all $n, k_1, \ldots, k_n \geq 1$, all $\ppi \in \Pi_n$, and all
$\ggamma^i \in \Pi_{k_i}$. 
\end{propn}

\begin{proof}
Write $\ggamma^i = \bigl(\gamma^i_1, \ldots, \gamma^i_{k_i}\bigr)$.  Choose
$a_i \in \Z$ representing $\pi_i \in \Zp$ and $b^i_j \in \Z$ representing
$\gamma^i_j \in \Zp$, for each $i$ and $j$.  Write $B^i = b^i_1 + \cdots +
b^i_{k_i}$.

We evaluate in turn the three terms in~(\ref{eq:chn-p-main}).
First, by Lemma~\ref{lemma:ent-partial} and the derivation property of
$\partial$ (Lemma~\ref{lemma:p-deriv-elem}),
\begin{align}
H\bigl( \ppi \of (\ggamma^1, \ldots, \ggamma^n) \bigr)     &
=
\sum_{i = 1}^n \sum_{j = 1}^{k_i} \partial\bigl(a_i b^i_j\bigr)
-
\partial \Biggl( \sum_{i = 1}^n \sum_{j = 1}^{k_i} a_i b^i_j \Biggr)    \\
&
=
\sum_{i = 1}^n \partial(a_i) B^i 
+ \sum_{i = 1}^n a_i \sum_{j = 1}^{k_i} \partial\bigl(b^i_j\bigr)
-
\partial \Biggl( \sum_{i = 1}^n a_i B^i \Biggr). 
\end{align}
Second, since $\ggamma^i \in \Pi_{k_i}$, we have $B^i \equiv 1 \pmod{p}$, so
$a_i B^i \in \Z$ represents $\ppi_i \in \Zp$.  Hence
\begin{align}
H(\ppi)      &
=
\sum_{i = 1}^n \partial\bigl(a_i B^i\bigr) 
- \partial\Biggl( \sum_{i = 1}^n a_i B^i \Biggr)        \\
&
=
\sum_{i = 1}^n \partial(a_i) B^i
+ \sum_{i = 1}^n a_i \partial(B^i)
- \partial\Biggl( \sum_{i = 1}^n a_i B^i \Biggr).        %\\
\end{align}
Third, 
\begin{align}
\sum_{i = 1}^n \pi_i H(\ggamma^i)       
=
\sum_{i = 1}^n a_i \sum_{j = 1}^{k_i} \partial\bigl(b^i_j\bigr)
- \sum_{i = 1}^n a_i \partial(B^i).
\end{align}
The result follows.
\end{proof}

A special case of composition is the \demph{tensor product} of
distributions, defined for $\ppi \in \Pi_n$ and $\ggamma \in \Pi_k$ by
\begin{align}
\ppi \otimes \ggamma    &
=
\ppi \of (\ggamma, \ldots, \ggamma)     \\
&
=
(\pi_1 \gamma_1, \ldots, \pi_1 \gamma_k,
\ \ldots, \ 
\pi_n \gamma_1, \ldots, \pi_n \gamma_k)
\in 
\Pi_{nk}.
\end{align}
In the analogous case of real distributions, $\ppi \otimes \ggamma$ is the
joint distribution of two independent random variables with distributions
$\ppi$ and $\ggamma$.

The chain rule immediately implies a logarithmic property of entropy
mod~$p$:

\begin{cor}
\lbl{cor:ent-log}
$H(\ppi \otimes \ggamma) = H(\ppi) + H(\ggamma)$ for all $\ppi \in
\Pi_n$ and $\ggamma \in \Pi_k$.
\qed
\end{cor}

\section{Unique characterization of entropy}
\lbl{sec:char-ent}

Our main theorem is that up to a constant factor, entropy mod~$p$ is the
only quantity satisfying the chain rule.

\begin{thm}
\lbl{thm:fad-p}
Let $\bigl( I \from \Pi_n \to \Zp \bigr)_{n \geq 1}$ be a sequence of
functions.  The following are equivalent:
\begin{enumerate}
\item 
\lbl{part:fad-p-condns}
$I$ satisfies the chain rule (that is, satisfies the conclusion of
Proposition~\ref{propn:chn-p} with $I$ in place of $H_p$);

\item
\lbl{part:fad-p-form}
$I = cH_p$ for some $c \in \Zp$.
\end{enumerate}
\end{thm}

Since $H$ satisfies the chain rule, so does any constant multiple.
Hence~\bref{part:fad-p-form} implies~\bref{part:fad-p-condns}.  We now
begin the proof of the converse.  

\emph{For the rest of the proof}, let
$\bigl( I \from \Pi_n \to \Zp \bigr)_{n \geq 1}$ be a sequence of functions
satisfying the chain rule.  Recall that $\vc{u}_n$ denotes the uniform
distribution $(1/n, \ldots, 1/n)$, for $p \ndvd n$.

\begin{lemma}
\lbl{lemma:ent-p-elem}
\begin{enumerate}
\item
\lbl{part:epe-bin}
$I(\vc{u}_{mn}) = I(\vc{u}_m) + I(\vc{u}_n)$ for all $m, n \in \nat$ not
  divisible by $p$;

\item 
\lbl{part:epe-null}
$I(\vc{u}_1) = 0$.
\end{enumerate}
\end{lemma}

\begin{proof}
By the chain rule, $I$  has the logarithmic property
\begin{align}
I(\ppi \otimes \ggamma) =
I\bigl(\ppi \of (\ggamma, \ldots, \ggamma)\bigr) =
I(\ppi) + I(\ggamma)
\end{align}
for all $\ppi \in \Pi_m$ and $\ggamma \in \Pi_n$.  In particular,
for all $m, n$ not divisible by $p$,
\begin{align}
I(\vc{u}_{mn}) 
= 
I(\vc{u}_m \otimes \vc{u}_n) 
=
I(\vc{u}_m) + I(\vc{u}_n),
\end{align}
proving~\bref{part:epe-bin}.  For~\bref{part:epe-null}, take $m = n = 1$
in~\bref{part:epe-bin}.
\end{proof}

\begin{lemma}
\lbl{lemma:fad-10}
$I(1, 0) = I(0, 1) = 0$.
\end{lemma}

\begin{proof}
% To prove that $I(1, 0) = 0$, w
We compute $I(1, 0, 0)$ in two ways.  On the
one hand, by the chain rule,
\begin{align}
I(1, 0, 0)
=
I\bigl((1, 0) \of \bigl((1, 0), \vc{u}_1\bigr)\bigr)
=
I(1, 0) + 1 \cdot I(1, 0) + 0 \cdot I(\vc{u}_1)
=
2I(1, 0).
\end{align}
On the other, by the chain rule and the fact that $I(\vc{u}_1) = 0$,
\begin{align}
I(1, 0, 0)
=
I\bigl((1, 0) \of \bigl(\vc{u}_1, (1, 0)\bigr)\bigr)
=
I(1, 0) + 1 \cdot I(\vc{u}_1) + 0 \cdot I(1, 0)
=
I(1, 0).
\end{align}
Hence $I(1, 0) = 0$.  The proof that $I(0, 1) = 0$ is similar.
\end{proof}

\begin{lemma}
\lbl{lemma:ent-p-abs}
For all $\ppi \in \Pi_n$ and $i \in \{0, \ldots, n\}$, 
\begin{align}
I(\pi_1, \ldots, \pi_n)
=
I(\pi_1, \ldots, \pi_i, 0, \pi_{i + 1}, \ldots, \pi_n).
\end{align}
\end{lemma}

\begin{proof}
First suppose that $i \neq 0$.  Then
\begin{align}
(\pi_1, \ldots, \pi_i, 0, \pi_{i + 1}, \ldots, \pi_n)
=
\ppi \of 
\bigl(\underbrace{\vc{u}_1, \ldots, \vc{u}_1}_{i - 1}, (1, 0), 
\underbrace{\vc{u}_1, \ldots, \vc{u}_1}_{n - i}\bigr).
\end{align}
Applying $I$ to both sides, then using the chain rule and $I(\vc{u}_1) = 0
= I(1, 0)$, gives the result.  The case $i = 0$ is proved similarly, 
using $I(0, 1) = 0$.
\end{proof}

We will prove the characterization theorem by analysing $I(\vc{u}_n)$ as
$n$ varies.  The chain rule will allow us to deduce the value of $I(\ppi)$
for more general distributions $\ppi$, thanks to the following lemma.

\begin{lemma}
\lbl{lemma:p-ufm-to-gen}
Let $\ppi \in \Pi_n$ with $\pi_i \neq 0$ for all $i$.  For each $i$, let
$k_i \geq 1$ be an integer representing $\pi_i \in \Zp$, and write $k =
\sum_{i = 1}^n k_i$.  Then 
\begin{align}
I(\ppi) = I(\vc{u}_k) - \sum_{i = 1}^n k_i I(\vc{u}_{k_i}).
\end{align}
\end{lemma}

\begin{proof}
First note that none of $k_1, \ldots, k_n, k$ is a multiple of $p$, so
$\vc{u}_{k_i}$ and $\vc{u}_k$ are well-defined.  We have
\begin{align}
\ppi \of (\vc{u}_{k_1}, \ldots, \vc{u}_{k_n}) 
= 
(\underbrace{1, \ldots, 1}_k)
=
\vc{u}_k.
\end{align}
Applying $I$ to both sides and using the chain rule gives the result.
\end{proof}

We come now to the most delicate part of the argument.  Since $H(\vc{u}_n)
= \fq{p}(n)$, and since $\fq{p}(n)$ is $p^2$-periodic in $n$, if $I$ is to
be a constant multiple of $H$ then $I(\vc{u}_n)$ must also be
$p^2$-periodic in $n$.  We show this directly.

\begin{lemma}
\lbl{lemma:p-per}
$I(\vc{u}_{n + p^2}) = I(\vc{u}_n)$ for all natural numbers $n$ not
  divisible by $p$.
\end{lemma}

\begin{proof}
First we prove the existence of a constant $c \in \Zp$ such that for all $n
\in \nat$ not divisible by $p$,
\begin{align}
\lbl{eq:per-one}
I(\vc{u}_{n + p}) = I(\vc{u}_n) - c/n.
\end{align}
(Compare Lemma~\ref{lemma:fq-elem}\bref{part:fq-elem-shift}.)  An
equivalent statement is that $n(I(\vc{u}_{n + p}) - I(\vc{u}_n))$ is
independent of $n \not\in p\nat$.  Since for any $n_1$ and $n_2$ we can
choose some $m \geq \max\{n_1, n_2\}$ with $m \equiv 1 \pmod{p}$, it is
enough to show that whenever $0 \leq n \leq m$ with $n \not\equiv 0
\pmod{p}$ and $m \equiv 1 \pmod{p}$,
\begin{align}
\lbl{eq:per-nm}
n \bigl( I(\vc{u}_{n + p}) - I(\vc{u}_{n}) \bigr)
=
I(\vc{u}_{m + p}) - I(\vc{u}_{m}).
\end{align}
To prove this, consider the distribution
\begin{align}
\ppi = (n, \underbrace{1, \ldots, 1}_{m - n}).
\end{align}
By Lemma~\ref{lemma:p-ufm-to-gen} and the fact that $I(\vc{u}_1) = 0$,
\begin{align}
I(\ppi) = I(\vc{u}_m) - nI(\vc{u}_n).
\end{align}
But also
\begin{align}
\ppi = (n + p, \underbrace{1, \ldots, 1}_{m - n}),
\end{align}
so by the same argument,
\begin{align}
I(\ppi) 
=
I(\vc{u}_{m + p}) - (n + p)I(\vc{u}_{n + p})    
= 
I(\vc{u}_{m + p}) - nI(\vc{u}_{n + p}).
\end{align}
Comparing the two expressions for $I(\ppi)$ gives
equation~\eqref{eq:per-nm}, thus proving the initial claim.

By induction on equation~\eqref{eq:per-one},
\begin{align}
I(\vc{u}_{n + rp}) = I(\vc{u}_n) - cr/n
\end{align}
for all $n, r \in \nat$ with $p \nmid n$.  The result follows
by putting $r = p$.
\end{proof}

We can now prove the characterization theorem for entropy modulo~$p$.

\begin{pfof}{Theorem~\ref{thm:fad-p}}
Define $f \from \{ n \in \nat \such p \ndvd n\} \to \Zp$ by $f(n) =
I(\vc{u}_n)$.  Lemma~\ref{lemma:ent-p-elem}, Lemma~\ref{lemma:p-per} and
Proposition~\ref{propn:fq-char} together imply that $f = c\fq{p}$ for some
$c \in \Zp$.  By Example~\ref{eg:ent-p-ufm}, an equivalent statement is
that $I(\vc{u}_n) = cH(\vc{u}_n)$ for all $n$ not divisible by $p$.

Since both $I$ and $cH$ satisfy the chain rule,
Lemma~\ref{lemma:p-ufm-to-gen} applies to both; and since $I$ and $cH$ are
equal on uniform distributions, they are also equal on all distributions
$\ppi$ such that $\pi_i \neq 0$ for all $i$.  Finally, applying
Lemma~\ref{lemma:ent-p-abs} to both $I$ and $cH$, we deduce by 
induction that $I(\ppi) = cH(\ppi)$ for all $\ppi \in \Pi_n$.
\end{pfof}

A variant of the characterization theorem will be useful.  The
distributions $(\pi_1, \ldots, \pi_n)$ considered so far can be viewed as
probability measures (mod~$p$) on sets of the form $\{1, \ldots, n\}$, but
it will be convenient to generalize to arbitrary finite sets.

Thus, given a finite set $X$, write $\Pi_X$ for the set of families
$\ppi = (\pi_x)_{x \in X}$ of elements of $\Zp$ such that $\sum_{x \in X}
\pi_x = 1$.  A \demph{finite probability space mod~$p$} is a finite set $X$
together with an element $\ppi \in \Pi_X$.  

As in the real case, we can take convex combinations of probability spaces.
Given a finite probability space $(X, \ppi)$ and a further family $(Y_x,
\ggamma^x)_{x \in X}$ of finite probability spaces, all mod~$p$, we obtain a new
probability space 
\begin{align}
\Biggl( \coprod_{x \in X} Y_x, \coprod_{x \in X} \pi_x \ggamma^x \Biggr)
\end{align}
mod~$p$.  Here $\coprod Y_x$ is the disjoint union of the sets $Y_x$, and
$\coprod \pi_x \ggamma^x$ gives probability $\pi_x \gamma^x_y$ to
an element $y \in Y_x$.

The operation of taking convex combinations is simply composition of
distributions, in different notation.  Indeed, if $X = \{1, \ldots, n\}$
and $Y_x = \{1, \ldots, k_x\}$ then the set $\coprod Y_x$ is naturally
identified with $\{1, \ldots, k_1 + \cdots + k_n\}$, and under this
identification, $\coprod \pi_x \ggamma^x$ corresponds to the composite
distribution $\ppi \of (\ggamma^1, \ldots, \ggamma^n)$.

The \demph{entropy} of $\ppi \in \Pi_X$ is, of course, defined as
\begin{align}
H(\ppi) = \frac{1}{p} \Biggl( 1 - \sum_{x \in X} a_x^p \Biggr),
\end{align}
where $a_x \in \Z$ represents $\pi_x \in \Zp$ for each $x \in X$.  It is
\demph{isomorphism-invariant}: whenever $(Y, \ssigma)$
and $(X, \ppi)$ are finite probability spaces mod~$p$ and there is some
bijection $f \from Y \to X$ satisfying $\sigma_y = \pi_{f(y)}$ for all $y
\in Y$, then $H(\ssigma) = H(\ppi)$.  The chain rule for entropy mod $p$,
translated into the notation of convex combinations, states that
\begin{align}
\lbl{eq:chn-cc}
H\Biggl( \coprod_{x \in X} \pi_x \ggamma^x \Biggr)
=
H(\ppi) + \sum_{x \in X} \pi_x H(\ggamma^x)
\end{align}
for all finite probability spaces $(X, \ppi)$ and $(Y_x, \ggamma^x)$ mod~$p$.

\begin{cor}
\lbl{cor:g-faddeev}
Let $I$ be a function assigning an element $I(\ppi)$ of $\Zp$ to each finite
probability space $(X, \ppi)$ mod~$p$.  The following are equivalent:
\begin{enumerate}
\item
\lbl{part:g-faddeev-condns}
$I$ is isomorphism-invariant and satisfies the chain
rule~\eqref{eq:chn-cc} (with $I$ in place of $H$); 

\item
\lbl{part:g-faddeev-form}
$I = cH$ for some $c \in \Zp$.
\end{enumerate}
\end{cor}

\begin{proof}
That~\bref{part:g-faddeev-form} implies~\bref{part:g-faddeev-condns}
follows from the observations above.
Conversely, take a function $I$ satisfying~\bref{part:g-faddeev-condns}.
Restricting $I$ to finite sets of the form $\{1, \ldots, n\}$ defines a
sequence of functions $(I \from \Pi_n \to \Zp)_{n \geq 1}$ satisfying the
chain rule.  By Theorem~\ref{thm:fad-p}, there is some constant $c \in \Zp$
such that $I(\ppi) = cH(\ppi)$ for all $n \geq 1$ and $\ppi \in \Pi_n$.
Now take any finite probability space $(Y, \ssigma)$.  We have
\begin{align}
(Y, \ssigma) 
\iso
\bigl( \{1, \ldots, n\}, \ppi \bigr)
\end{align}
for some $n \geq 1$ and $\ppi \in \Pi_n$, and by
isomorphism-invariance of both $I$ and $H$,
\begin{align}
I(\ssigma)
=
I(\ppi)
=
cH(\ppi)
=
cH(\ssigma),
\end{align}
proving~\bref{part:g-faddeev-form}.
\end{proof}

\begin{remark}
This corollary is slightly weaker than our main characterization result,
Theorem~\ref{thm:fad-p}.  Indeed, if $I$ is an isomorphism-invariant
function on the class of finite probability spaces mod~$p$ then in
particular, permuting the arguments of a measure does not change the
value that $I$ gives it.  Thus, the corollary also follows from a weaker
version of Theorem~\ref{thm:fad-p} in which the putative entropy
function is also assumed to be symmetric in its arguments.  But
Theorem~\ref{thm:fad-p} shows that the symmetry assumption is, in fact,
unnecessary.
\end{remark}

\section{Information loss}
\lbl{sec:loss}

\begin{quote}%
Grothendieck came along and said, `No, the
  Riemann--Roch theorem is \emph{not} a theorem about varieties, it's a
  theorem about morphisms between varieties.'
\hfill 
---Nicholas Katz (quoted in \cite{JackCAN}, p.~1046).
\end{quote}

The entropy of a probability space is a special case of a more general
concept, the information loss of a map between probability spaces.  This
point is most easily explained through the real case, as follows.

Given a real probability distribution $\ppi$ on a finite set, the entropy
of $\ppi$ is the amount of information gained by learning the result of an
observation drawn from $\ppi$.  For example, if $\ppi = (1/4, 1/4, 1/4,
1/4)$ then the entropy (to base $2$) is $2$, reflecting the fact that
results of draws from $\ppi$ cannot be communicated in fewer than $2$ bits
each.

In the same spirit, one can ask how much information is lost by a
deterministic process.  Consider, for instance, the process of forgetting
the suit of a card drawn fairly from a standard $52$-card pack.  Since the
four suits are distributed uniformly, $2$ bits of information are lost.  An
alternative viewpoint is that the information loss is the amount of
information at the start of the process minus the amount at the end, which
is
\begin{align}
H(1/52, \ldots, 1/52) - H(1/13, \ldots, 1/13)
=
\log 52 - \log 13
=
\log 4.
\end{align}
If we take logarithms to base $2$ then the information loss is, again, $2$
bits.  Hence the two viewpoints give the same result.

Generally, given a measure-preserving map $f \from (Y, \ssigma) \to (X,
\ppi)$ between finite probability spaces, we can quantify the information
lost by $f$ in either of two equivalent ways.  We can condition on the
outcome $x$, taking for each $x$ the amount of information lost by
collapsing the fibre $f^{-1}(x)$:
\begin{align}
\lbl{eq:loss-cond}
\sum_{x \csuch \pi_x \neq 0} 
\pi_x
H \Biggl( 
\biggl(\frac{\sigma_y}{\pi_x}\biggr)_{y \in f^{-1}(x)}
\Biggr).
\end{align}
(The argument of $H$ is the distribution $\ssigma$ restricted to
$f^{-1}(x)$ and normalized to sum to $1$.)  Alternatively, we can subtract
the amount of information at the end of the process from the amount at the
start:
\begin{align}
\lbl{eq:loss-diff}
H(\ssigma) - H(\ppi).
\end{align}
The two expressions~\eqref{eq:loss-cond} and~\eqref{eq:loss-diff} are
equal, as we will show in the analogous mod $p$ case.

Entropy is the special case of information loss where one discards
\emph{all} the information.  That is, the entropy of a probability
distribution $\ssigma$ on a set $Y$ is the information loss of the unique
map from $(Y, \ssigma)$ to the one-point space.  In this sense, the concept
of information loss subsumes the concept of entropy.

The description so far is of information loss over $\R$, which was analysed
and characterized in Baez, Fritz and Leinster~\cite{CETIL}.  (In
particular, equation~(5) of~\cite{CETIL} describes the relationship between
information loss and conditional entropy.)  We now show that a strictly
analogous characterization theorem holds over $\Zp$, even in the absence of
an information-theoretic interpretation.

\begin{defn}
Let $(Y, \ssigma)$ and $(X, \ppi)$ be finite probability spaces mod $p$.  A
\demph{measure-preserving map} $f \from (Y, \ssigma) \to (X, \ppi)$ is a
function $f \from Y \to X$ such that
\begin{align}
\pi_x = \sum_{y \in f^{-1}(x)} \sigma_y
\end{align}
for all $x \in X$.  
\end{defn}

Finite probability spaces mod~$p$ and their measure-preserving maps form a
category $\FinProbp$.  The construction of convex combinations is
functorial, in the following sense: given a finite probability space $(X,
\ppi)$ mod~$p$ and a family of maps
\begin{align}
\Bigl( (Y_x, \ssigma^x) \toby{f_x} (Z_x, \ttau^x) \Bigr)_{x \in X}
\end{align}
in $\FinProbp$, we have the map
\begin{align}
\coprod_{x \in X} \pi_x f_x \from 
\Biggl( \coprod_{x \in X} Y_x, \coprod_{x \in X} \pi_x \ssigma^x \Biggr)
\to
\Biggl( \coprod_{x \in X} Z_x, \coprod_{x \in X} \pi_x \ttau^x \Biggr)
\end{align}
in $\FinProbp$ that maps $y \in Y_x$ to $f_x(y) \in Z_x$.  (Although the
function $\coprod \pi_x f_x$ does not depend on $\ppi$ and would usually be
written as just $\coprod f_x$, it will be convenient to use this more
informative notation.)

Entropy is an invariant of the objects of $\FinProbp$, and
information loss is an invariant of the maps in $\FinProbp$:

\begin{defn}
Let $f \from (Y, \ssigma) \to (X, \ppi)$ be a measure-preserving map
between finite probability spaces mod~$p$.  The \demph{information loss} of
$f$ is
\begin{align}
L(f) = H(\ssigma) - H(\ppi) \in \Zp.
\end{align}
\end{defn}

\begin{lemma}
Let $f \from (Y, \ssigma) \to (X, \ppi)$ be a measure-preserving map
between finite probability spaces mod~$p$.  Then
\begin{align}
L(f)
=
\sum_{x \csuch \pi_x \neq 0} 
\pi_x
H \Biggl( 
\biggl(\frac{\sigma_y}{\pi_x}\biggr)_{y \in f^{-1}(x)}
\Biggr).
\end{align}
\end{lemma}

\begin{proof}
Since entropy is unaffected by adjoining elements of probability $0$, we
may assume that $\pi_x \neq 0$ for each $x \in X$.  Write $\ggamma^x$ for
the probability distribution $(\sigma_y/\pi_x)_{y \in f^{-1}(x)}$ on the
set $f^{-1}(x)$.  The probability space $(Y, \ssigma)$ mod~$p$ is 
isomorphic to
\begin{align}
\Biggl( 
\coprod_{x \in X} f^{-1}(x), 
\coprod_{x \in X} \pi_x \ggamma^x
\Biggr),
\end{align}
and the chain rule~\eqref{eq:chn-cc} then gives
\begin{align}
H(\ssigma) = H(\ppi) + \sum_{x \in X} \pi_x H(\ggamma^x).
\end{align}
\end{proof}

Information loss has some intuitively reasonable properties.  First, an
invertible process loses no information: $L(f) = 0$ whenever $f$ is an
isomorphism in $\FinProbp$.  This follows from the isomorphism-invariance
of entropy.

Second, the information loss of two processes performed in series is the
sum of the information lost by each individually:
\begin{align}
L(g \of f) = L(g) + L(f)
\end{align}
for any maps
\begin{align}
\lbl{eq:series}
(Y, \ssigma) \toby{f} (X, \ppi) \toby{g} (W, \nnu)
\end{align}
in $\FinProbp$.  This is immediate from the definition.

Third, the information loss of a convex combination of two
processes performed in parallel is the corresponding convex combination of
their individual information losses.  That is, given
$\lambda \in \Zp$ and maps
\begin{eqnarray*}
(Y, \ssigma) &\toby{f}& (Z, \ttau),   \\
(Y', \ssigma') &\toby{f'}& (Z', \ttau')
\end{eqnarray*}
in $\FinProbp$, we have
\begin{align}
\lbl{eq:parallel-loss}
L(\lambda f \cpd (1 - \lambda) f') = \lambda L(f) + (1 - \lambda) L(f').
\end{align}
Indeed, using the chain rule~\eqref{eq:chn-cc} and writing $\lambda' = 1 -
\lambda$, 
\begin{align}
L(\lambda f \cpd \lambda' f')      &
=
H(\lambda \ssigma \cpd \lambda' \ssigma')
-
H(\lambda \ttau \cpd \lambda' \ttau')        \\
&
=
\Bigl\{ 
H(\lambda, \lambda') + \lambda H(\ssigma) + \lambda' H(\ssigma') 
\Big\} 
-
\Bigl\{ 
H(\lambda, \lambda') + \lambda H(\ttau) + \lambda' H(\ttau') 
\Big\},
\end{align}
and equation~\eqref{eq:parallel-loss} follows.

These three properties of information loss mod~$p$ are enough to
characterize it completely, up to a constant factor.

\begin{thm}
\lbl{thm:loss}
Let $K$ be a function assigning an element $K(f)$ of $\Zp$ to each
measure-preserving map $f$ between finite probability spaces mod~$p$.  The
following are equivalent:
\begin{enumerate}
\item 
\lbl{part:loss-condns}
$K$ has these three properties:
\begin{enumerate}
\item 
\lbl{part:lc-iso}
$K(f) = 0$ for all isomorphisms $f$;

\item
\lbl{part:lc-series}
$K(g \of f) = K(g) + K(f)$ for all composable pairs~\eqref{eq:series} of
  measure-preserving maps;

\item
\lbl{part:lc-parallel}
$K\bigl(\lambda f \cpd (1 - \lambda) f'\bigr) = \lambda K(f) + (1 -
  \lambda) K(f')$ for all measure-preserving maps $f$ and $f'$ and all
  $\lambda \in \Zp$;
\end{enumerate}

\item
\lbl{part:loss-form}
$K = cL$ for some $c \in \Zp$.
\end{enumerate}
\end{thm}

\begin{remark}
Like any group, $\Zp$ can be regarded as a one-object category, and
conditions~\hardref{(a)} and~\hardref{(b)} then imply that $K$ is a functor
$\FinProbp \to \Zp$.  
\end{remark}

\begin{pfof}{Theorem~\ref{thm:loss}}
We have already shown that information loss $L$ satisfies the three
conditions of~\bref{part:loss-condns}, and it follows
that~\bref{part:loss-form} implies~\bref{part:loss-condns}.

For the converse, suppose that $K$ satisfies~\bref{part:loss-condns}.
Given a finite probability space $(X, \ppi)$, write $!_{\ppi}$ for the unique
measure-preserving map
\begin{align}
!_{\ppi} \from (X, \ppi) \to (\{1\}, \vc{u}_1),
\end{align}
and define $I(\ppi) = K(!_{\ppi})$.  For any measure-preserving map $f \from
(Y, \ssigma) \to (X, \ppi)$, the triangle
\begin{align}
\xymatrix{
(Y, \ssigma) \ar[rr]^f \ar[rd]_{!_{\ssigma}}      &
&
(X, \ppi) \ar[ld]^{!_{\ppi}}        \\
&
(\{1\}, \vc{u}_1)
}
\end{align}
commutes, so by condition~\hardref{(b)}, 
\begin{align}
\lbl{eq:K-I-diff}
K(f) = I(\ssigma) - I(\ppi).
\end{align}
So in order to prove the theorem, it suffices to show that $I = cH$ for
some constant $c$.  And for this, it is enough to prove that $I$
satisfies the hypotheses of Corollary~\ref{cor:g-faddeev}.

First, $I$ is isomorphism-invariant, since if $f\from (Y, \ssigma) \to (X,
\ppi)$ is an isomorphism then $K(f) = 0$, so $I(\ssigma) = I(\ppi)$
by~\eqref{eq:K-I-diff}.  

Second, $I$ satisfies the chain rule~\eqref{eq:chn-cc}; that is,
\begin{align}
\lbl{eq:cetil-ch}
I\Biggl( \coprod_{x \in X} \pi_x \ggamma^x \Biggr)
=
I(\ppi) + \sum_{x \in X} \pi_x I(\ggamma^x)
\end{align}
for all finite probability spaces $(X, \ppi)$ and $(Y_x, \ggamma^x)$
mod~$p$.  To see this, write
\begin{align}
f \from
\coprod_{x \in X} Y_x \to X
\end{align}
for the function defined by $f(y) = x$ whenever $y \in Y_x$.  Then $f$
defines a measure-preserving map
\begin{align}
f \from
\Biggl( \coprod_{x \in X} Y_x, \coprod \pi_x \ggamma^x \Biggr)
\to
(X, \ppi)
\end{align}
We now evaluate $K(f)$ in two ways.  On the one hand, by
equation~\eqref{eq:K-I-diff},
\begin{align}
K(f) =
I\Bigl( \coprod \pi_x \ggamma^x \Bigr) - I(\ppi).
\end{align}
On the other, 
\begin{align}
f = \coprod \pi_x \, !_{\ggamma^x},
\end{align}
so by condition~\hardref{(c)} and induction,
\begin{align}
K(f)
=
\sum \pi_x K(!_{\ggamma^x})
=
\sum \pi_x I(\ggamma^x).
\end{align}
Comparing the two expressions for $K(f)$ gives the chain rule 
(equation~\eqref{eq:chn-cc}) for $I$, as claimed.

Corollary~\ref{cor:g-faddeev} therefore applies, giving $I = cH$ for some
$c \in \Zp$.  It follows from equation~\eqref{eq:K-I-diff} that $K = cL$.
\end{pfof}

Theorem~\ref{thm:loss} has two striking features.  First, the main
equations that characterize information loss,
\begin{align}
L(g \of f) = L(g) + L(f),
\qquad
L(\lambda f \cpd (1 - \lambda) f') 
=
\lambda L(f) + (1 - \lambda) f',
\end{align}
are entirely linear.  Despite the fact that information loss subsumes
entropy, the equations are simpler in form than the characterizing equation
for entropy, the chain rule.

A second striking feature of Theorem~\ref{thm:loss} is that the axioms
on the hypothetical information loss function $K$ force
$K(f)$ to depend only on the domain and codomain of $f$.  This is an
instance of a general categorical fact: for a functor $K$ from a category
$\cat{P}$ with a terminal object to a groupoid, $K(f) = K(f')$
whenever $f$ and $f'$ are maps in $\cat{P}$ with the same domain and the same
codomain.

\section{The residue mod $p$ of real entropy}
\lbl{sec:res}

At the end of the note~\cite{KontOHL} in which he initiated the
subject of entropy modulo a prime, Kontsevich wrote:
\begin{quote}
\emph{Conclusion:} If we have a random variable $\xi$ which takes
finitely many values with all probabilities in $\Q$ then we can define
not only the transcendental number $H(\xi)$ but also its `residues
modulo $p$' for almost all primes $p$\,!
\end{quote}
Formally, given $n \geq 1$ and a prime $p$, write $\Rat{n}{p}$ for the set
of finite probability distributions $\ppi = (\pi_1, \ldots, \pi_n)$ where
each $\pi_i$ is a rational number expressible as a fraction with
denominator not divisible by $p$.  Then each $\ppi \in \Rat{n}{p}$
represents an element of $\Pi_n$, and the suggestion is to view $H_p(\ppi)
\in \Zp$ as the residue mod~$p$ of the real number $H_\R(\ppi)$.

Although the quotation above was the sum total of what Kontsevich wrote on
the matter, his suggestion can be developed.  First, different
distributions can have the same entropy over $\R$; for instance,
\begin{align}
H_\R(1/2, 1/8, 1/8, 1/8, 1/8) 
=
H_\R(1/4, 1/4, 1/4, 1/4).
\end{align}
There is, therefore, a question of consistency: Kontsevich's proposal only
makes sense if
\begin{align}
H_\R(\ppi) = H_\R(\ggamma) \implies H_p(\ppi) = H_p(\ggamma)
\end{align}
for all $\ppi \in \Rat{n}{p}$ and $\ggamma \in \Rat{m}{p}$.  Second, the
word `residue' suggests additivity: that the residue of a sum should be the
sum of the residues.  

We will show that both these properties are indeed satisfied: there is a
well-defined, addition-preserving map
\begin{align}
\begin{array}{ccc}
\bigcup_{n = 1}^\infty 
\{ H_\R(\ppi) \such \ppi \in \Rat{n}{p} \}      &
\to     &
\Zp,    \\
H_\R(\ppi)      &\mapsto        &H_p(\ppi).
\end{array}
\end{align}
\begin{lemma}
\lbl{lemma:p-prod-par}
Let $n, m \geq 1$ and let $a_1, \ldots, a_n, b_1, \ldots, b_m \geq 0$ be
integers.  Then
\begin{align}
\prod_{i = 1}^n a_i^{a_i} = \prod_{j = 1}^m b_j^{b_j}
\implies
\sum_{i = 1}^n \partial_p(a_i) = \sum_{j = 1}^m \partial_p(b_j),
\end{align}
where the first equality is in $\Z$, the second is in $\Zp$, and we set
$0^0 = 1$. 
\end{lemma}

\begin{proof}
Since $0^0 = 1$ and $\partial(0) = 0$, it is enough to prove the result in
the case where each of the integers $a_i$ and $b_j$ is strictly positive.
We may then write $a_i = p^{\alpha_i} A_i$ with $\alpha_i \geq 0$ and $p
\ndvd A_i$, and similarly $b_j = p^{\beta_j} B_j$.  We adopt the convention
that by default, the index $i$ ranges over $1, \ldots, n$ and the index $j$
over $1, \ldots, m$.

Assume that $\prod a_i^{a_i} = \prod b_j^{b_j}$.  We have
\begin{align}
\prod a_i^{a_i}
=
p^{\sum \alpha_i a_i} \prod A_i^{a_i}
\end{align}
with $p \ndvd \prod A_i^{a_i}$, and similarly for $\prod b_j^{b_j}$.
It follows that
\begin{align}
\prod A_i^{a_i}         &
=
\prod B_j^{b_j},        
\lbl{eq:ppp-prod}       \\
\sum \alpha_i a_i &
=
\sum \beta_j b_j.
\lbl{eq:ppp-sum}        
\end{align}
We consider each of these equations in turn.

First, since $p \ndvd \prod A_i^{a_i}$, the Fermat quotient
$\fq{p}\bigl(\prod A_i^{a_i}\bigr)$ is well-defined, and the logarithmic
property of $\fq{p}$ (Lemma~\ref{lemma:fq-elem}\bref{part:fq-elem-log})
gives 
\begin{align}
-\fq{p}\Bigl( \prod A_i^{a_i} \Bigr)
=
\sum - a_i \fq{p}(A_i).
\end{align}
Consider the right-hand side as an element of $\Zp$.  When $p \dvd a_i$,
the $i$-summand vanishes.  When $p \ndvd a_i$, the $i$-summand is $-a_i
\fq{p}(a_i) = \partial(a_i)$.  Hence
\begin{align}
-\fq{p}\Bigl( \prod A_i^{a_i} \Bigr)
=
\sum_{i \csuch \alpha_i = 0} \partial(a_i)
\end{align}
in $\Zp$.  A similar result holds for $\prod B_j^{b_j}$, so
equation~\eqref{eq:ppp-prod} gives
\begin{align}
\lbl{eq:ppp-zero}
\sum_{i \csuch \alpha_i = 0} \partial(a_i)
=
\sum_{j \csuch \beta_j = 0} \partial(b_j).
\end{align}

Second,
\begin{align}
\sum_{i = 1}^n \alpha_i a_i 
= 
\sum_{i \csuch \alpha_i \geq 1} \alpha_i a_i,
\end{align}
so $p \dvd \sum \alpha_i a_i$.  Now
\begin{align}
\frac{1}{p} \sum \alpha_i a_i 
=
\sum_{i \csuch \alpha_i \geq 1} \alpha_i p^{\alpha_i - 1} A_i
\equiv
\sum_{i \csuch \alpha_i = 1} A_i \pmod{p},
\end{align}
and if $\alpha_i = 1$ then $A_i = a_i/p = \partial(a_i)$.  
A similar result holds for $\sum \beta_j b_j$, so
equation~\eqref{eq:ppp-sum} gives
\begin{align}
\lbl{eq:ppp-one}
\sum_{i \csuch \alpha_i = 1} \partial(a_i)
=
\sum_{j \csuch \beta_j = 1} \partial(b_j)
\end{align}
in $\Zp$.

Finally, for each $i$ such that $\alpha_i \geq 2$, we have $p^2 \dvd a_i$
and so $\partial(a_i) = 0$ in $\Zp$.  Hence
\begin{align}
\lbl{eq:ppp-two}
\sum_{i \csuch \alpha_i \geq 2} \partial(a_i)
=
\sum_{j \csuch \beta_j \geq 2} \partial(b_j),
\end{align}
both sides being $0$.  Summing equations~\eqref{eq:ppp-zero},
\eqref{eq:ppp-one} and~\eqref{eq:ppp-two} gives the result.
\end{proof}

We deduce that the real entropy of a rational distribution determines its
entropy modulo $p$:

\begin{thm}
\lbl{thm:R-p}
Let $n, m \geq 1$, $\ppi \in \Rat{n}{p}$ and $\ggamma \in
\Rat{m}{p}$.  Then
\begin{align}
H_\R(\ppi) = H_\R(\ggamma) \implies H_p(\ppi) = H_p(\ggamma).
\end{align}
\end{thm}

\begin{proof}
We can write
\begin{align}
\ppi = (r_1/t, \ldots, r_n/t),  
\qquad
\ggamma = (s_1/t, \ldots, s_m/t),
\end{align}
where $r_i$, $s_j$ and $t$ are nonnegative integers with $p \ndvd t$ and 
\begin{align}
r_1 + \cdots + r_n = t = s_1 + \cdots + s_m.
\end{align}
By multiplying all of these integers by a constant, we may assume that $t
\equiv 1 \pmod{p}$.

We have
\begin{align}
e^{-H_\R(\ppi)} = \prod_i (r_i/t)^{r_i/t},
\end{align}
with the convention that $0^0 = 1$.  Multiplying both sides by $t$ then
raising to the power of $t$ gives
\begin{align}
t^t e^{-t H_\R(\ppi)} = \prod_i r_i^{r_i}.
\end{align}
By the analogous equation for $\ggamma$ and the assumption that $H_\R(\ppi)
= H_\R(\ggamma)$, it follows that
\begin{align}
\prod_i r_i^{r_i} = \prod_j s_j^{s_j}.
\end{align}
By Lemma~\ref{lemma:p-prod-par}, then,
\begin{align}
\sum_i \partial(r_i) = \sum_j \partial (s_j)
\end{align}
in $\Zp$.  Moreover, $\sum r_i = t = \sum s_j$.  Hence
\begin{align}
\sum_i \partial(r_i) - \partial \Biggl( \sum_i r_i \Biggr)
=
\sum_j \partial(s_j) - \partial \Biggl( \sum_j s_j \Biggr).
\end{align}
But $t \equiv 1 \pmod{p}$, so $r_i$ represents the element $r_i/t = \pi_i$
of $\Zp$, so by Lemma~\ref{lemma:ent-partial}, the left-hand side of this
equation is $H_p(\ppi)$.  Similarly, the right-hand side is
$H_p(\ggamma)$. Hence $H_p(\ppi) = H_p(\ggamma)$.
\end{proof}

It follows that Kontsevich's residue classes of real entropies are
well-defined.  That is, writing
\begin{align}
\Rentp
=
\bigcup_{n = 1}^\infty
\bigl\{ H_\R(\ppi) \such \ppi \in \Rat{n}{p} \bigr\}
\sub \R,
\end{align}
there is a unique map of sets
\begin{align}
[\,\cdot\,] \from \Rentp \to \Zp
\end{align}
such that $[H_\R(\ppi)] = H_p(\ppi)$ for all $\ppi \in
\Rat{n}{p}$ and $n \geq 1$.  

\begin{propn}
The set $\Rentp$ is closed under addition, and the residue map
\begin{align}
{}[\,\cdot\,] \from     
\Rentp
\to     
\Zp
\end{align}
preserves addition.
\end{propn}

\begin{proof}
Let $\ppi \in \Rat{n}{p}$ and $\ggamma \in \Rat{m}{p}$.  We must
show that $H_\R(\ppi) + H_\R(\ggamma) \in \Rentp$ and 
\begin{align}
[H_\R(\ppi) + H_\R(\ggamma)] 
=
[H_\R(\ppi)] + [H_\R(\ggamma)].
\end{align}
We will use the tensor product of real probability distributions, which is
defined by the same formula as for distributions over $\Zp$
(Section~\ref{sec:properties}).  Evidently $\ppi \otimes \ggamma \in
\Rat{nm}{p}$, and it is an instance of the chain rule that
\begin{align}
H_\R(\ppi \otimes \ggamma)
=
H_\R(\ppi) + H_\R(\ggamma).  
\end{align}
Hence $H_\R(\ppi) + H_\R(\ggamma) \in \Rentp$, and 
\begin{align}
[H_\R(\ppi) + H_\R(\ggamma)]    
&
= 
[H_\R(\ppi \otimes \ggamma)]    
=
H_p(\ppi \otimes \ggamma)       \\
&
=
H_p(\ppi) + H_p(\ggamma)        
=
[H_\R(\ppi)] + [H_\R(\ggamma)],
\end{align}
where the third equality is by Corollary~\ref{cor:ent-log}.  
\end{proof}

\begin{remark}
The set $\Rentp$ appears to have no very simple description.  Evidently it
is an additive submonoid of the $\Q$-linear subspace of $\R$ with basis $\{
\log \ell \such \text{primes } \ell\}$.  One can show that $\log \ell \in
\Rentp$ for each prime $\ell \neq p$ and that $\log p \not\in \Rentp$.
However, some elements of $\Rentp$ do contain components of $\log p$.
\end{remark}

\section{Entropy as a polynomial}
\lbl{sec:poly-defn}

There is an alternative approach to entropy modulo a prime.  Previously, to
define the entropy of a distribution mod $p$, we had to step outside $\Zp$
to make arbitrary choices of integers representing the `probabilities',
then show that the definition was independent of those choices
(Definition~\ref{defn:ent-p}).  We now show how to define $H(\ppi)$
directly as a function of $\pi_1, \ldots, \pi_n$.  That function is a
polynomial, by the following classical fact:

\begin{lemma}
\lbl{lemma:fn-fin-fld}
Let $\fld$ be a finite field with $q$ elements, let $n \geq 0$, and let $F
\from \fld^n \to \fld$ be a function.  Then there is a unique polynomial
$f$ of the form
\begin{align}
\lbl{eq:poly-small-deg}
f(x_1, \ldots, x_n)
=
\sum_{0 \leq r_1, \ldots, r_n < q} c_{r_1, \ldots, r_n} 
x_1^{r_1} \cdots x_n^{r_n}
\end{align}
($c_{r_1, \ldots, r_n} \in \fld$) such that
\begin{align}
f(\pi_1, \ldots, \pi_n) = F(\pi_1, \ldots, \pi_n)
\end{align}
for all $\pi_1, \ldots, \pi_n \in \fld$.
\end{lemma}

\begin{proof}
Write $\fld^{< q}[x_1, \ldots, x_n]$ for the set of polynomials of the
form~\eqref{eq:poly-small-deg}.  Write $R(f) \from \fld^n \to \fld$ for the
function induced by a polynomial $f$ in $n$ variables.  Then $R$ defines a
map
\begin{align}
R \from \fld^{< q}[x_1, \ldots, x_n] \to 
\{\text{functions } \fld^n \to \fld\}.
\end{align}
We have to prove that $R$ is bijective.  Both domain and codomain have
$q^{q^n}$ elements, so it suffices to prove that $R$ is surjective.

First define a polynomial $\delta$ by 
\begin{align}
\delta(x_1, \ldots, x_n) = (1 - x_1^{q - 1}) \cdots (1 - x_n^{q - 1}).
\end{align}
For $a_1, \ldots, a_n \in K$,
\begin{align}
R(\delta)(a_1, \ldots, a_n) 
=
\begin{cases}
1       &\text{if } a_1 = \cdots = a_n = 0,     \\
0       &\text{otherwise}.
\end{cases}
\end{align}
Now, given a function $F \from K^n \to K$, define a polynomial $f$ by
\begin{align}
f(x_1, \ldots, x_n)
=
\sum_{a_1, \ldots, a_n \in K} 
F(a_1, \ldots, a_n) \delta(x_1 - a_1, \ldots, x_n - a_n).
\end{align}
Then $f \in \fld^{< q}[x_1, \ldots, x_n]$ and $R(f) = F$.
\end{proof}

In particular, taking $\fld = \Zp$, entropy modulo~$p$ can be expressed as
a polynomial of degree less than $p$ in each variable.  For each $n \geq
1$, define $h(x_1, \ldots, x_n) \in (\Zp)[x_1, \ldots, x_n]$ by
\begin{align}
h(x_1, \ldots, x_n) 
=
-\sum_{\substack{0 \leq r_1, \ldots, r_n < p\\ r_1 + \cdots + r_n = p}}
\frac{x_1^{r_1} \cdots x_n^{r_n}}{r_1! \cdots r_n!}.
\end{align}

\begin{propn}
\lbl{propn:ent-p-eqv}
For all $n \geq 1$ and $\ppi \in \Pi_n$,
\begin{align}
H(\pi_1, \ldots, \pi_n) = h(\pi_1, \ldots, \pi_n).
\end{align}
\end{propn}

\begin{proof}
Let $\pi_1, \ldots, \pi_n \in \Zp$.  We will show that whenever $a_1,
\ldots, a_n$ are integers representing $\pi_1, \ldots, \pi_n$, then
\begin{align}
\lbl{eq:Hh}
\frac{1}{p} \Biggl( 
\Biggl( \sum_{i = 1}^n a_i \Biggr)^p - \sum_{i = 1}^n a_i^p
\Biggr)
\end{align}
is an integer representing $h(\pi_1, \ldots, \pi_n)$.  The result will
follow, since if $\ppi \in \Pi_n$ then $\sum \pi_i = 1$, so $(\sum a_i)^p
\equiv 1 \pmod{p^2}$ by Lemma~\ref{lemma:pps}.

We have to prove that
\begin{align}
\Biggl( \sum_{i = 1}^n a_i \Biggr)^p - \sum_{i = 1}^n a_i^p 
\equiv
-p
\sum_{\substack{0 \leq r_1, \ldots, r_n < p\\ r_1 + \cdots + r_n = p}}
\frac{a_1^{r_1} \cdots a_n^{r_n}}{r_1! \cdots r_n!}
\pmod{p^2}.
\end{align}
Since $(p - 1)!$ is invertible in $\Zps$, an equivalent statement is that
\begin{align}
\lbl{eq:ent-p-eqv-1}
(p - 1)! \Biggl( 
\sum_{i = 1}^n a_i^p  - \Biggl( \sum_{i = 1}^n a_i \Biggr)^p
\Biggr)
\equiv
\sum_{\substack{0 \leq r_1, \ldots, r_n < p\\ r_1 + \cdots + r_n = p}}
\frac{p!}{r_1! \cdots r_n!}
a_1^{r_1} \cdots a_n^{r_n}
\pmod{p^2}.
\end{align}
The right-hand side of~\eqref{eq:ent-p-eqv-1} is $\bigl(\sum a_i\bigr)^p -
\sum a_i^p$, so equation~\eqref{eq:ent-p-eqv-1} reduces to 
\begin{align}
\bigl( (p - 1)! + 1 \bigr)
\Biggl( 
\sum_{i = 1}^n a_i^p - \Biggl( \sum_{i = 1}^n a_i \Biggr)^p
\Biggr)
\equiv
0
\pmod{p^2}.
\end{align}
And since $(p - 1)! \equiv - 1 \pmod{p}$ and $\sum a_i^p \equiv \sum a_i
\equiv \bigl(\sum a_i\bigr)^p \pmod{p}$, this is true. 
\end{proof}

\begin{remark}
The polynomial $h$ is homogeneous of degree $p$, so the induced function on
$(\Zp)^n$ is a degree $1$ homogeneous extension $\bar{H} \from (\Zp)^n \to
\Zp$ of the entropy function $H \from \Pi_n \to \Zp$.  

Everything that we have done for $H$ can also be done for $\bar{H}$.
Equation~\eqref{eq:Hh} expresses $\bar{H}(\ppi)$ in terms of integers $a_i$
representing its arguments.  As in Lemma~\ref{lemma:ent-partial},
$\bar{H}(\ppi)$ can equivalently be expressed as $\sum \partial(a_i) -
\partial\bigl(\sum a_i\bigr)$.  The definition and characterization of
information loss can be extended to finite \demph{measure spaces mod~$p$}
(sets $X$ equipped with an element of $(\Zp)^X$), and the convexity
condition~\eqref{eq:parallel-loss} is then replaced by linearity
conditions: $L(f \cpd f') = L(f) + L(f')$ and $L(\lambda f) = \lambda
L(f)$.  An analogous characterization theorem over $\R$ was already proved
as Corollary~4 of~\cite{CETIL}.
\end{remark}

\section{Polynomial identities satisfied by entropy}
\lbl{sec:poly-ids}

We now establish further polynomial identities in $h$, stronger than the
functional equations previously proved for $H$.  The first is closely
related to the chain rule, as we shall see.

\begin{thm}
\lbl{thm:p-grouping}
Let $n, k_1, \ldots, k_n \geq 0$.  Then $h$ satisfies the following
identity of polynomials in commuting variables $y_{ij}$ over $\Zp$:
\begin{multline*}
h(y_{11}, \ldots, y_{1k_1}, 
\ \ldots, \ 
y_{n1}, \ldots, y_{nk_n})
\\
=
h(y_{11} + \cdots + y_{1k_1}, \ \ldots, \
y_{n1} + \cdots + y_{nk_n})
+
\sum_{i = 1}^n h(y_{i1}, \ldots, y_{ik_i}).
\end{multline*}
\end{thm}

\begin{proof}
The left-hand side is equal to
\begin{align}
\lbl{eq:gpg-1}
-\sum_{\substack{0 \leq s_1, \ldots, s_n \leq p\\ s_1 + \cdots + s_n = p}}
\sum
\frac{y_{11}^{r_{11}} \cdots y_{1k_1}^{r_{1k_1}} \ \cdots \ 
y_{n1}^{r_{n1}} \cdots y_{nk_n}^{r_{nk_n}}}%
{r_{11}! \cdots r_{1k_1}! \ \cdots \ r_{n1}! \cdots r_{nk_n}!},
\end{align}
where the inner sum is over all $0 \leq r_{ij} < p$ such that
\begin{align}
r_{11} + \cdots + r_{1k_1} = s_1,
\ \ldots,\ 
r_{n1} + \cdots + r_{nk_n} = s_n.
\end{align}
Split the outer sum into two parts, the
first consisting of the summands in which none of $s_1, \ldots, s_n$ is
equal to $p$, and the second consisting of the summands in which one $s_i$
is equal to $p$ and the others are zero.  Then the
polynomial~\eqref{eq:gpg-1} is equal to $A + B$, where
\begin{align}
A       &
=
-\sum_{\substack{0 \leq s_1, \ldots, s_n < p\\ s_1 + \cdots + s_n = p}}
\ 
\prod_{i = 1}^n
\sum_{\substack{r_{i1}, \ldots, r_{ik_i} \geq 0\\ 
r_{i1} + \cdots + r_{ik_i} = s_i}}
\frac{y_{i1}^{r_{i1}} \cdots y_{ik_i}^{r_{ik_i}}}%
{r_{i1}! \cdots r_{ik_i}!},     \\
B       &
=
-\sum_{i = 1}^n 
\ 
\sum_{\substack{0 \leq r_{i1}, \ldots, r_{ik_i} < p\\ 
r_{i1} + \cdots + r_{ik_i} = p}}
\frac{y_{i1}^{r_{i1}} \cdots y_{ik_i}^{r_{ik_i}}}%
{r_{i1}! \cdots r_{ik_i}!}.
\end{align}
We have
\begin{align}
A       &
=
-\sum_{\substack{0 \leq s_1, \ldots, s_n < p\\ s_1 + \cdots + s_n = p}}
\frac{1}{s_1! \cdots s_n!}
\prod_{i = 1}^n
\sum_{\substack{r_{i1}, \ldots, r_{ik_i} \geq 0\\ 
r_{i1} + \cdots + r_{ik_i} = s_i}}
\frac{s_i!}{r_{i1}! \cdots r_{ik_i}!}      
y_{i1}^{r_{i1}} \cdots y_{ik_i}^{r_{ik_i}}  \\
&
=
-\sum_{\substack{0 \leq s_1, \ldots, s_n < p\\ s_1 + \cdots + s_n = p}}
\frac{1}{s_1! \cdots s_n!}
\prod_{i = 1}^n
(y_{i1} + \cdots + y_{ik_i})^{s_i}      \\
&
=
h(y_{11} + \cdots + y_{1k_1}, \ \ldots,\ 
y_{n1} + \cdots + y_{nk_n})
\end{align}
and
\begin{align}
B = 
\sum_{i = 1}^n h(y_{i1}, \ldots, y_{ik_i}).
\end{align}
The result follows.
\end{proof}

\begin{cor}[Polynomial chain rule]
\lbl{cor:p-chn-poly}
Let $n, k_1, \ldots, k_n \geq 0$.  Then $h$ satisfies the following
identity of polynomials in commuting variables $x_i$, $y_{ij}$ over $\Zp$:
\begin{multline*}
h(x_1 y_{11}, \ldots, x_1 y_{1k_1}, 
\ \ldots, \ 
x_n y_{n1}, \ldots, x_n y_{nk_n})
\\
=
h\bigl( x_1(y_{11} + \cdots + y_{1k_1}), \ \ldots, \
x_n(y_{n1} + \cdots + y_{nk_n}) \bigr)
+
\sum_{i = 1}^n x_i^p h(y_{i1}, \ldots, y_{ik_i}).
\end{multline*}
\end{cor}

\begin{proof}
This follows from Theorem~\ref{thm:p-grouping} on substituting $x_i y_{ij}$
for $y_{ij}$, using the fact that $h$ is homogeneous of degree $p$.
\end{proof}

The original chain rule for entropy mod~$p$ (Proposition~\ref{propn:chn-p})
follows: given $\ppi \in \Pi_n$ and $\ggamma^i \in \Pi_{k_i}$ as in that
proposition, substitute $x_i = \pi_i$ and $y_{ij} = \gamma^i_j$.

The entropy polynomial $h(x)$ in one variable is $0$, by definition.  But
the entropy polynomial in two variables is nontrivial and satisfies a
cocycle condition:

\begin{cor}
\lbl{cor:cocycle}
The two-variable entropy polynomial $h$ satisfies the polynomial identity
\begin{align}
h(x, y) - h(x, y + z) + h(x + y, z) - h(y, z) = 0.
\end{align}
\end{cor}

Similar results appear in Cathelineau~\cite{CathSHS} (p.~58--59),
Kontsevich~\cite{KontOHL}, and Elbaz--Vincent and Gangl~\cite{EVGFPM}
(Section~2.3). 

\begin{proof}
Theorem~\ref{thm:p-grouping} with $n = 2$ and $(k_1, k_2) = (2, 1)$ gives
\begin{align}
h(x, y, z) = h(x + y, z) + h(x, y),
\end{align}
and similarly,
\begin{align}
h(x, y, z) = h(x, y + z) + h(y, z).
\end{align}
The result follows.
\end{proof}

We are especially interested in the case where the arguments of the entropy
function sum to $1$.  Under that restriction, $h(x, y)$ reduces to a simple
expression:

\begin{propn}
\lbl{propn:p-to-two}
If $p \neq 2$, there is an identity of polynomials
\begin{align}
h(x, 1 - x) = \sum_{r = 1}^{p - 1} \frac{x^r}{r},
\end{align}
and if $p = 2$, there is an identity of polynomials 
\begin{align}
h(x, 1 - x) = x + x^2.
\end{align}
\end{propn}

\begin{proof}
The case $p = 2$ is trivial; suppose otherwise.  In Example~\ref{eg:p-bin},
we proved the equality of functions
\begin{align}
h(\pi, 1 - \pi) = \sum_{r = 1}^{p - 1} \frac{\pi^r}{r}
\end{align}
($\pi \in \Zp$).  We now have to prove that this is a \emph{polynomial}
identity.  By Lemma~\ref{lemma:fn-fin-fld}, it suffices to show that the
polynomial
\begin{align}
h(x, 1 - x) 
= 
-\sum_{r = 1}^{p - 1} \frac{x^r(1 - x)^{p - r}}{r!(p - r)!}
\end{align}
has degree strictly less than $p$.  Since it plainly has degree at most
$p$, we only need to show that the coefficient of $x^p$ vanishes.  

The coefficient of $x^p$ in $h(x, 1 - x)$ is
\begin{align}
-\sum_{r = 1}^{p - 1} \frac{(-1)^{p - r}}{r! (p - r)!}.
\end{align}
For $1 \leq r \leq p - 1$,
\begin{align}
- \frac{(-1)^{p - r}}{r! (p - r)!}
=
(-1)^{p - r} \frac{(p - 1)!}{r! (p - r)!}
=
(-1)^{p - r} \frac{1}{r} \binom{p - 1}{r - 1}
=
(-1)^{p - 1} \frac{1}{r}
\end{align}
in $\Zp$, using first the fact that $(p - 1)! = -1$ and then
Lemma~\ref{lemma:p-binom}.  Hence the coefficient of $x^p$ in $h(x, 1 - x)$
is $(-1)^{p - 1} \sum_{r = 1}^{p - 1} 1/r$.  But $r \mapsto 1/r$ defines a
permutation of $(\Zp)^\times$, so the sum is equal to $\sum_{r = 1}^{p
  - 1} r$, which is $0$ since $p$ is odd.
\end{proof}

Following Elbaz-Vincent and Gangl~\cite{EVGOPI}, we write
\begin{align}
\lbl{eq:sterling}
\pounds_1(x) 
= 
h(x, 1 - x)
=
\begin{cases}
\sum_{r = 1}^{p - 1} x^r/r      &\text{if } p \neq 2,   \\
x + x^2                         &\text{if } p = 2.
\end{cases}
\end{align}
(Elbaz-Vincent and Gangl assumed that $p \neq 2$.) Despite the lack of
formal resemblance, $\pounds_1$ is the mod~$p$ analogue of the
real function
\begin{align}
\lbl{eq:real-two}
x 
\mapsto 
H_\R(x, 1 - x)
=
- x \log x - (1 - x) \log(1 - x).
\end{align}

Since $h$ is evidently a symmetric polynomial, 
\begin{align}
\lbl{eq:pounds-sym}
\pounds_1(x) = \pounds_1(1 - x)
\end{align}
in $(\Zp)[x]$.  The polynomial $\pounds_1$ also satisfies a more
complicated identity whose significance will be explained shortly.
Following Kontsevich~\cite{KontOHL}, Elbaz-Vincent and Gangl proved:
 
\begin{propn}[Elbaz-Vincent and Gangl]
\lbl{propn:p-feith}
There is a polynomial identity
\begin{align}
\pounds_1(x) + (1 - x)^p \pounds_1\biggl( \frac{y}{1 - x} \biggr)
=
\pounds_1(y) + (1 - y)^p \pounds_1\biggl( \frac{x}{1 - y} \biggr).
\end{align}
\end{propn}

Both sides of this equation are indeed polynomials, as $\deg(\pounds_1)
\leq p$.  Elbaz-Vincent and Gangl proved it using differential equations
(Proposition~5.9(2) of~\cite{EVGOPI}), but it also follows easily from the
cocycle identity for $h$:

\begin{proof}
Since $h$ is homogeneous of degree $p$, 
\begin{align}
h(x, y) = (x + y)^p \pounds_1\biggl( \frac{x}{x + y} \biggr).
\end{align}
The identity to be proved is, therefore, equivalent to
\begin{align}
h(x, 1 - x) + h(y, 1 - x - y)
=
h(y, 1 - y) + h(x, 1 - x - y).
\end{align}
Since $h$ is symmetric, this in turn is equivalent to
\begin{align}
h(x, 1 - x - y) - h(x, 1 - x) + h(1 - y, y) - h(1 - x - y, y) = 0,
\end{align}
which is an instance of the cocycle identity
of Corollary~\ref{cor:cocycle}. 
\end{proof}

Proposition~\ref{propn:p-feith} can be understood as follows.  Any
finite probability distribution can be expressed as an iterated composite
of distributions on two elements.  Hence, using the chain rule, the
entropy of any distribution can be computed in terms of entropies of
distributions on two elements.  In this sense, the sequence of functions
$\bigl( H \from \Pi_n \to \Zp \bigr)_{n \geq 1}$ reduces to the single
function $H \from \Pi_2 \to \Zp$, which is effectively a function in one
variable:
\begin{align}
\begin{array}{cccc}
F\from  &\Zp    &\to            &\Zp,   \\
        &\pi    &\mapsto        &H(\pi, 1 - \pi).
\end{array}
\end{align}
A similar reduction can be performed over $\R$.  

On the other hand, an arbitrary function $F \from \Zp \to \Zp$ cannot
generally be extended to a sequence of functions $\Pi_n \to \Zp$ satisfying
the chain rule (nor, similarly, in the real case).  Indeed, by expressing a
distribution $(\pi, 1 - \pi - \tau, \tau)$ as a composite in two different
ways, we obtain an equation that $F$ must satisfy if such an extension is
to exist. Assuming the symmetry property $F(\pi) = F(1 - \pi)$, that
equation is
\begin{align}
\lbl{eq:feith}
F(\pi) + (1 - \pi) F\biggl(\frac{\tau}{1 - \pi}\biggr)
=
F(\tau) + (1 - \tau) F\biggl(\frac{\pi}{1 - \tau}\biggr)
\end{align}
($\pi, \tau \neq 1$); compare Proposition~\ref{propn:p-feith}.

Equation~\eqref{eq:feith} is sometimes called the `fundamental equation of
information theory'.  Thus, Proposition~\ref{propn:p-feith} is a polynomial
version mod $p$ of the fundamental equation.  Over $\R$, it has been
studied since at least 1958~\cite{Tver}.  Assuming that $F$ is symmetric,
the fundamental equation is the only obstacle to the extension problem, in
the sense that if $F$ satisfies~\eqref{eq:feith} then the extension can be
performed.

In the real case, the function~\eqref{eq:real-two} is a solution of the
fundamental equation.  Up to a scalar multiple, it is the \emph{only}
measurable solution $F$ of the fundamental equation satisfying $F(0) =
F(1)$.  It can be deduced that up to a constant factor, Shannon entropy for
finite real probability distributions is characterized uniquely by
measurability, symmetry and the chain rule (Lee~\cite{Lee}).

In the mod~$p$ case, we know that the function $F = \pounds_1$ is symmetric
and satisfies the fundamental equation.  Since any such function $F$ can be
extended to a sequence of functions $\Pi_n \to \Zp$ satisfying the chain
rule, it follows from Theorem~\ref{thm:fad-p} that up to a constant factor,
$\pounds_1$ is the unique symmetric solution of the fundamental equation.

\begin{remark}
\lbl{rmk:k}
In his seminal note~\cite{KontOHL}, Kontsevich unified the real and
mod~$p$ cases with a homological argument, using a cocycle identity
equivalent to that in Corollary~\ref{cor:cocycle}.  In doing so, he
established that $\sum_{0 < r < p} \pi^r/r$ is the correct formula for the
entropy mod~$p$ of a distribution $(\pi, 1 - \pi)$ mod~$p$ on two elements
(assuming, as he did, that $p \neq 2$).  Although he gave no definition of
the entropy of a probability distribution mod $p$ on an arbitrary finite
number of elements, his arguments showed that a unique reasonable such
definition must exist.

The present work develops the framework hinted at in~\cite{KontOHL}, and
provides the further definition and characterization of information loss
mod~$p$.  It also makes two improvements to~\cite{KontOHL}.

The first is the
streamlined inclusion of the case $p = 2$.  
The second is the dropping of all symmetry requirements.  In axiomatic
approaches to entropy based on the fundamental equation of information
theory~\eqref{eq:feith}, such as those of Lee~\cite{Lee} and Kontsevich,
the symmetry axiom $F(\pi) = F(1 - \pi)$ is essential.  Indeed, $F(\pi) =
\pi$ is also a solution of~\eqref{eq:feith}, and similarly, the polynomial
identity of Proposition~\ref{propn:p-feith} is also satisfied by $x^p$ in
place of $\pounds_1(x)$.  The symmetry axiom is used to rule out these and
other undesired solutions.  This is why Lee's characterization of real
entropy needed the assumption that it is symmetric in its arguments.  In
contrast, symmetry is needed nowhere in the approach that we have taken.
\end{remark}

\bibliography{mathrefs}

\begin{thebibliography}{10}

\bibitem{AposIAN}
T.~M. Apostol.
\newblock {\em Introduction to Analytic Number Theory}.
\newblock Undergraduate Texts in Mathematics. Springer, 1976.

\bibitem{CETIL}
J.~Baez, T.~Fritz, and T.~Leinster.
\newblock A characterization of entropy in terms of information loss.
\newblock {\em Entropy}, 13:1945--1957, 2011.

\bibitem{BaBe}
P.~Baudot and D.~Bennequin.
\newblock The homological nature of entropy.
\newblock {\em Entropy}, 17:3253--3318, 2015.

\bibitem{BuiuDCA}
A.~Buium.
\newblock Differential characters of abelian varieties over $p$-adic fields.
\newblock {\em Inventiones Mathematicae}, 122:309--340, 1995.

\bibitem{CathSHS}
J.-L. Cathelineau.
\newblock Sur l'homologie de $\mathrm{SL}_2$ {\`a} coefficients dans l'action
  adjointe.
\newblock {\em Mathematica Scandinavica}, 63:51--86, 1988.

\bibitem{CathRDP}
J.-L. Cathelineau.
\newblock Remarques sur les diff{\'e}rentielles des polylogarithmes uniformes.
\newblock {\em Annales de l'Institut {F}ourier}, 46:1327--1347, 1996.

\bibitem{DenipEp}
C.~Deninger.
\newblock $p$-adic entropy and a $p$-adic {F}uglede--{K}adison determinant.
\newblock In Y.~Tschinkel and Y.~Zarhin, editors, {\em Algebra, Arithmetic, and
  Geometry}, volume 269 of {\em Progress in Mathematics}, pages 423--442.
  Birkh{\"a}user, Boston, 2009.

\bibitem{Eise}
G.~Eisenstein.
\newblock Neue {G}attung zahlentheoretischen {F}unktionen, die von zwei
  {E}lementen abh{\"a}ngen und durch gewisse lineare {F}unktional-{G}leichungen
  definirt werden.
\newblock {\em Bericht {\"u}ber die zur {B}ekanntmachung geeigneten
  {V}erhandlungen der {K}{\"o}niglich {P}reussischen {A}kademie der
  {W}issenschaften zu {B}erlin}, pages 36--42, 1850.

\bibitem{EVGOPI}
P.~Elbaz-Vincent and H.~Gangl.
\newblock On poly(ana)logs {I}.
\newblock {\em Compositio Mathematica}, 130:161--214, 2002.

\bibitem{EVGFPM}
P.~Elbaz-Vincent and H.~Gangl.
\newblock Finite polylogarithms, their multiple analogues and the {S}hannon
  entropy.
\newblock In F.~Nielsen and F.~Barbaresco, editors, {\em Geometric Science of
  Information 2015}, volume 9389 of {\em Lecture Notes in Computer Science},
  pages 277--285. Springer, 2015.

\bibitem{Fadd}
D.~K. Faddeev.
\newblock On the concept of entropy of a finite probabilistic scheme (in
  {R}ussian).
\newblock {\em Uspekhi Matematicheskikh Nauk}, 11:227--231, 1956.

\bibitem{JackCAN}
A.~Jackson.
\newblock \emph{Comme appel{\'e} du n{\'e}ant}---as if summoned from the void:
  the life of {A}lexandre {G}rothendieck.
\newblock {\em Notices of the American Mathematical Society}, 51(9):1038--1056,
  2004.

\bibitem{JoyDAL}
A.~Joyal.
\newblock $\delta$-anneaux et $\lambda$-anneaux.
\newblock {\em Comptes Rendus Math{\'e}matiques de l'Acad{\'e}mie des Sciences
  (Canada)}, 7:227--232, 1985.

\bibitem{KontOHL}
M.~Kontsevich.
\newblock The $1\tfrac{1}{2}$-logarithm.
\newblock Private note, 1995.
\newblock Reprinted as appendix of~\cite{EVGOPI}.

\bibitem{Lee}
P.~M. Lee.
\newblock On the axioms of information theory.
\newblock {\em The Annals of Mathematical Statistics}, 35:415--418, 1964.

\bibitem{Tver}
H.~Tverberg.
\newblock A new derivation of the information function.
\newblock {\em Mathematica Scandinavica}, 6:297--298, 1958.

\bibitem{VignTSS}
J.~P. Vigneaux.
\newblock {\em Topology of statistical systems: a cohomological approach to
  information theory}.
\newblock PhD thesis, Universit{\'e} Paris Diderot, 2019.

\end{thebibliography}

\end{document}